\documentclass[UTF8,oneside,10pt]{article}
\usepackage[a4paper,margin=1.3in]{geometry}
\usepackage{amsmath}           % AMS 的主包
\usepackage{amssymb}           % AMS 符号包，会自动加载 amsfonts
\usepackage{mathrsfs}          % \mathscr 字体样式
\usepackage{eucal}             % 更改 \mathcal 的字体样式
\usepackage{amsthm}            % AMS 的定理环境
\usepackage{mathtools}         % 一些额外的数学环境功能，包括 \dfrac
\usepackage{stmaryrd}          % 更更多的特殊符号
\usepackage{arcs}              % 弧線hat
\usepackage{esint}             % 更多积分符号
\usepackage{extarrows}         % 提供在箭头上下写字的命令
\usepackage{enumerate}         % 带序号列表环境
\usepackage{xcolor}            % 让 latex 支持花里胡哨的颜色的包
\usepackage{graphicx}          % 插入各种类型图片支持
\usepackage{tikz}              % tikz 绘图环境
\usepackage{tikz-3dplot}       % 一个简单的 tikz 3d 绘图宏包
\usepackage{tikz-cd}           % 基于 tikz 的交换图表绘制工具
\usepackage[all,cmtip]{xy}     % 交换图表绘制工具，命令为 \xymatrix
\usepackage{pgfplots}          % 高级绘图包，基于 tikz，包括 2d 和 3d 绘图
\pgfplotsset{compat=newest}    % 设置 pgfplots 兼容性选项，以使用新功能
\usepackage{subcaption}        % 可以交叉引用的图表并排环境
\usepackage[ocgcolorlinks,linkcolor=blue]{hyperref}
\usepackage{cleveref}          
\usepackage{tcolorbox}         % 绘制彩色文本框的宏包
\tcbuselibrary{most}           % 加载 tcolorbox 的库
\usepackage{bm}                % 为所有数学字体添加粗体，命令 \bm{abcd}
\usepackage{slashed}           % 在符号上添加反划线，命令 \slashed

\usepackage[utf8]{inputenc}
\usepackage[T1]{fontenc}

\usepackage{svg}               % 插入svg圖

\usepackage{multicol}          %用于实现在同一页中实现不同的分栏

\usepackage{amsmath,amsthm,amsfonts,amssymb,bm}

\usepackage{fourier}

\usepackage{bbold}

\usepackage{graphicx} %插入图片的宏包
\usepackage{float} %设置图片浮动位置的宏包

\theoremstyle{plain}\newtheorem{theorem}{Theorem}[section]
\theoremstyle{definition}\newtheorem{definition}[theorem]{Definition}
\theoremstyle{plain}
\theoremstyle{plain}\newtheorem{corollary}[theorem]{Corollary}
\theoremstyle{plain}\newtheorem{lemma}[theorem]{Lemma}
\theoremstyle{plain}
\theoremstyle{plain}\newtheorem{proposition}[theorem]{Proposition}
\theoremstyle{plain}
\theoremstyle{plain}
\theoremstyle{plain}
\theoremstyle{remark}
\theoremstyle{remark}
\theoremstyle{remark}
\theoremstyle{remark}
\theoremstyle{plain}
\theoremstyle{plain}
\theoremstyle{plain}
\theoremstyle{remark}
\theoremstyle{remark}
\theoremstyle{remark}
\numberwithin{equation}{section}
\numberwithin{theorem}{section}

\usepackage[backend=bibtex]{biblatex}
\usepackage{xcolor}

\definecolor{c1}{HTML}{88d5d2} %淡蓝
\definecolor{c2}{HTML}{9c9d47} %橄榄绿
\definecolor{c3}{HTML}{fec842} %鹅蛋黄
\definecolor{c4}{HTML}{e97a2e} %霞光黄
\definecolor{c5}{HTML}{834e71} %月光紫

\usepackage{xcolor}
\definecolor{myred}{HTML}{a43113}
\definecolor{myblue}{HTML}{1f4887}
\definecolor{mygreen}{HTML}{49804c}
\definecolor{mygraygreen}{HTML}{6b6c5c}
\definecolor{myyellow}{HTML}{d5934a}
\definecolor{mypurple}{HTML}{492f7a}

\begin{document}
% \maketitle

\begin{center}
{\Large \textbf{Minimal Filling $K$-Systems of Curves}}

{\Large \textbf{ }} 

{\large \textsc{Hong} CHANG$^{*}$, \textsc{Xiao} CHEN$^{**}$ and \textsc{Wujie} SHEN$^{***}$}

{\footnotesize \emph{e-mail:} changhong@pku.edu.cn$^{*}$}

{\footnotesize \emph{e-mail:} x-chen20@tsinghua.org.cn$^{**}$}

{\footnotesize \emph{e-mail:} shenwj22@mails.tsinghua.edu.cn$^{***}$}

\end{center}

\begin{abstract}

    In this paper, we determine the exact minimal number of curves in a filling $k$-system on an oriented surface of genus $g$ for any positive integers $k$ and $g$.

\end{abstract}

% \keywords{A; B; C.}

% \newpage
%\tableofcontents
% \newpage

%==============================%
%           正文内容           %
%------------------------------%

% 图片格式：
% \begin{figure}[H]
% \centering
% \includegraphics[width=0.9\textwidth]{Pictures//HalfTwist.jpg}
% \caption{A half twist.}
% \label{fig: HalfTwist}
% \end{figure}

\section{Introduction}\label{section: Introduction}

%\subsection{Filling multicurves}

Let $S_g$ be a closed orientable surface of genus $g$. For a positive integer $k$, a $k$-system of curves is a finite collection of closed curves (we also call it a \emph{multicurve}) with pairwise intersections at most $k$, as introduced in \cite{MR1405702} for the purpose of studying graph embeddings on surfaces. The study of $k$-systems is closely related to that of the curve complex \cite{harvey1981boundary}, which has been a recurring theme in surface topology and mapping class groups \cite{MR3209702,zbMATH01089297,MR1722024}. Throughout this paper, we work in the smooth category, and a "curve" means a simple closed curve, and we assume all intersections between curves are transverse.%For further investigations about $k$-systems, we may refer the readers to \cite{chen2024systems,1ChenShen2024,MR4034921,przytycki2015arcs}.

A multicurve in $S_g$ is \emph{filling} if its complement is a union of discs and it is in minimal position. Note that once a multicurve fills and is in minimal position, it remains filling in all isotopy classes, so in this paper we always treat curves as isotopy classes of curves. Such multicurves play a key role in Thurston's construction of pseudo-Anosov diffeomorphisms \cite{Thur} and its generalisation by Penner \cite{Penner}. When there are two curves, they are called a \emph{filling pair} and correspond to a pair of vertices in the curve complex that are at distance at least $3$.

Among the various studies, a natural problem is to construct filling multicurves in minimal position, and in particular to construct them explicitly. In \cite{AH} Aougab and Huang gave a sharp lower bound on the intersection number of filling pairs on $S_g$ and provided explicit constructions achieving the bound. Then Jefferys, Aougab--Menasco--Nieland and Chang--Menasco gave various constructions of minimal filling pairs in \cite{AMN,CM,Jeff}.

%If the total intersection number achieves the minimal number of intersections for all filling union of curves on a given surface, then it is said to be {\em minimally intersecting}, or minimal. Various studies of it arises and one of the natural problems is the construction of filling multicurves, and in particular, the construction of minimal filling multicurves. When there are two curves, they are called a \emph{filling pair} and correspond to a pair of vertices in the curve complex that are a distance 3 or greater. Aougab--Huang proved the existence of minimal filling pairs in \cite{AH}. Then Jefferys, Aougab--Menasco--Nieland and Chang--Menasco gave various constructions of minimal filling pairs in \cite{Jeff}, \cite{AMN} and \cite{CM} individually.

For general multicurves, the maximal size of filling multicurves has been studied; see, for example, \cite{AACLOSX,AG,PS,KP}. In this paper, we prove Theorem~\ref{thm:main2}, which gives the minimum possible size of a filling $k$-system and provides an explicit construction achieving this minimum. This result follows as a corollary of Theorem~\ref{thm:main1}, which determines the maximal genus of a surface admitting a filling $k$-system of $n$ curves.

\begin{theorem}\label{thm:main1}
    Let $k\ge1,n\ge 2$ be positive integers such that \(n \geq 4\), or \(n = 3\) and \(k \neq 1\), or $n=2$ and $k\not=3$. Among all surfaces $S_g$ admitting a filling $k$-system of $n$ curves, the maximal possible genus $g$ is 
    \begin{equation}\label{eq:gkn}
    g_{k,n}=\left\lfloor \frac{1}{4} k n (n-1) + \frac{1}{2} \right\rfloor,
    \end{equation}where $\lfloor m\rfloor$ is the greatest integer less than or equal to $m$. Moreover, $g_{1,3}=g_{3,2}=1$.
\end{theorem}

\begin{theorem}\label{thm:main2}
The minimal number of curves in a filling \(k\)-system on \(S_g\) is
\begin{equation}\label{equ:n0}
n_0 = n_0(k,g) := \max\left\{\left\lceil \frac{1}{2} \left( 1 + \sqrt{\frac{16g-8}{k} + 1} \right) \right\rceil,2 \right\},
\end{equation}
when $g\geq3$, or $g=2$ and $k\neq1,3$, where $\lceil m\rceil$ is the least integer greater than or equal to $m$. Moreover, $n_0(1,2)=4$ and $n_0(3,2)=3$. In other words, suppose a collection of essential curves \(\gamma_1,\dots,\gamma_n\) intersect pairwise at most \(k\) times, and \(S_g \setminus \bigl( \bigcup_{i=1}^{n} \gamma_i \bigr)\) is a union of disks. Then the minimum possible number \(n\) is \(n_0\) as defined above.
\end{theorem}

The case $n=2$ follows from Lemma 2.1 and Theorem 2.17 of \cite{AH}. For a filling pair $(\alpha,\beta)$ on $S_g$, the intersection number is at least $2g-1$, and on $S_2$ it is at least $4$. Consequently, when $n=2$, the maximal genus of a surface admitting a filling pair with intersection number at most $k$ is $\lfloor\frac{k+1}{2}\rfloor$ for $k\ge 4$, and $1$ for $k=1,2,3$. In this paper, we only deal with the cases in which $n\ge 3$.

The paper is structured as follows. In Section~\ref{section: Background} we recall some notation concerning systems of curves on surfaces, for the sake of rigour and completeness. The main ingredient of the proof is the thickened multicurve introduced in Section~\ref{subsec:thickenedmulticurve}. Using this, we prove the upper bound for \(g\) in Theorem~\ref{thm:main1} and the lower bound for \(n\) in Theorem~\ref{thm:lowerbound}. In the remaining sections, we prove that the upper bound in Theorem~\ref{thm:main1} and the lower bound in Theorem~\ref{thm:main2} are sharp. Specifically, we construct a filling $k$-system consisting of $n$ curves on a surface of maximal genus $g_{k,n}$. The construction is given separately for odd and even $k$ in Section \ref{sec:kodd}. However, to complete the proof of Theorem~\ref{thm:main2}, we must also construct a filling $k$-system of $n$ curves on $S_g$ for any genus $g$ with $g_{k,n-1}<g<g_{k,n}$, this is carried out in Section \ref{sec:completion2}.

\section{Background}\label{section: Background}

\subsection{Curves on surfaces}

\begin{definition}[Curves and multicurves]
A \emph{simple closed curve}, or a \emph{simple loop}, on a surface $S$ is defined by an embedding $\gamma : \mathbb{S}^1 \hookrightarrow S$. A collection of curves on $S$ is called a \emph{multicurve}. Throughout this paper, we use the term \emph{curve} to mean a simple closed curve, and we always assume that the curves in a multicurve intersect pairwise \emph{transversely}, i.e. they intersect transversely at all intersection points. We will often conflate a (multi)curve with its image.
\end{definition}

\begin{definition}[Regular neighbourhoods of curves]
\label{definition RN}
Let $L$ be a multicurve on $S$. A regular neighbourhood $W(L)$ of $L$ is a locally flat, closed subsurface of $S$ containing $\cup_{\gamma \in L} \gamma$ such that there is a strong deformation retraction $H: W(L) \times I \rightarrow W(L)$ onto $\cup_{\gamma \in L} \gamma$ where $H\vert_{(W(L) \cap {\partial S}) \times I}$ is a strong deformation retraction onto $\cup_{\gamma \in L} \gamma \cap {\partial S}$.
\end{definition}

\begin{definition}[Free isotopies]
Let $\gamma_0$ and $\gamma_1$ be two curves. We say $\gamma_0$ is freely isotopic to $\gamma_1$ if there exists a continuous map $\tilde{\gamma}: [0,1] \times \mathbb{S}^1 \rightarrow S$ such that $\tilde{\gamma}(0,\mathbb{S}^1)$ and $\tilde{\gamma}(1,\mathbb{S}^1)$ represent $\gamma_0$ and $\gamma_1$, respectively. We denote the \emph{free isotopy class} of $\gamma$ by $[\gamma]$.
%Let $\gamma_0 , \gamma_1$ be two curves. We say $\gamma_0$ is freely isotopic to $\gamma_1$, if there is a family of curves $\left\{ \tilde{\gamma}_t \middle| t \in [0,1] \right\}$ such that $\tilde{\gamma}_0 = \gamma_0 , \tilde{\gamma}_1 = \gamma_1$ and $\tilde{\gamma}: [0,1] \times \mathbb{S}^1 \rightarrow S$ is a continuous map. We denote the \emph{free isotopy class} by $H := [\gamma]$.
\end{definition}

\begin{definition}[Essential curves]
We say that a simple curve is \emph{essential} if it is not homotopically trivial, cannot be isotoped into any boundary component, and is primitive.
\end{definition}

\subsection{Systems of curves}

\begin{definition}[Geometric intersection numbers]
Given two curves $\gamma_1, \gamma_2: \mathbb{S}^1 \rightarrow S$, we define the \emph{geometric intersection number} $i (\gamma_1 , \gamma_2)$ as
\begin{align}
i (\gamma_1 , \gamma_2) := \# \left\{ (z_1, z_2) \in \mathbb{S}^1 \times \mathbb{S}^1 \middle| \gamma_1(z_1) = \gamma_2(z_2) \right\} . \notag
\end{align}
Consider two free isotopy classes (not necessarily distinct) $[\gamma_1] , [\gamma_2]$ of two simple curves $\gamma_1, \gamma_2$ on $S$. We define the \emph{geometric intersection number} of $[\gamma_1]$ and $[\gamma_2]$ as follows:
\begin{align}
i ([\gamma_1] , [\gamma_2]) := \min \left\{ i (\bar{\gamma}_1 , \bar{\gamma}_2) \middle|  \bar{\gamma}_1 \in [\gamma_1] , \bar{\gamma}_2 \in [\gamma_2] \right\}. \notag
\end{align}
\end{definition}

\begin{definition}[Systems of curves]
\label{definition: Systems of curves}
A \emph{system} $L$ of curves on $S$ is a multicurve, such that the curves in $L$ are essential, any two distinct curves $\gamma_1, \gamma_2$ are non-isotopic, and lie in minimal position (that is $i(\gamma_1, \gamma_2) = i([\gamma_1], [\gamma_2])$), see \cite{MR2850125}.
\end{definition}

\begin{definition}[$k$-systems of curves]
\label{definition: $k$-systems of curves}
We call a system of curves a \emph{$k$-system of curves} on $S$ if the geometric intersection number of any pair of elements in the system is at most $k$. Let $\mathscr{L}(S,k)$ be the collection of all $k$-systems of curves on $S$.
\end{definition}

\begin{definition}[Complete $k$-systems of curves]
\label{definition: Complete $k$-systems of curves}
We call a system of curves a \emph{complete $k$-system of curves} on $S$ if the geometric intersection number of any pair of elements in the system is exactly $k$. We let $\widehat{\mathscr{L}}(S,k)$ be the collection of all complete $k$-systems of curves on $S$.
\end{definition}

\begin{definition}[Filling systems of curves]
\label{definition: Filling systems of curves}
    A system of curves in a compact surface $S$ is \emph{filling} if their complement in $S$ is a collection of disks. In Section 3 of \cite{AMN}, the authors provided an example of a filling system, which consists of two curves (see \cref{fig:filling pair}).
\end{definition}

\begin{figure} [H]
    \centering
    \includegraphics[width=1\linewidth]{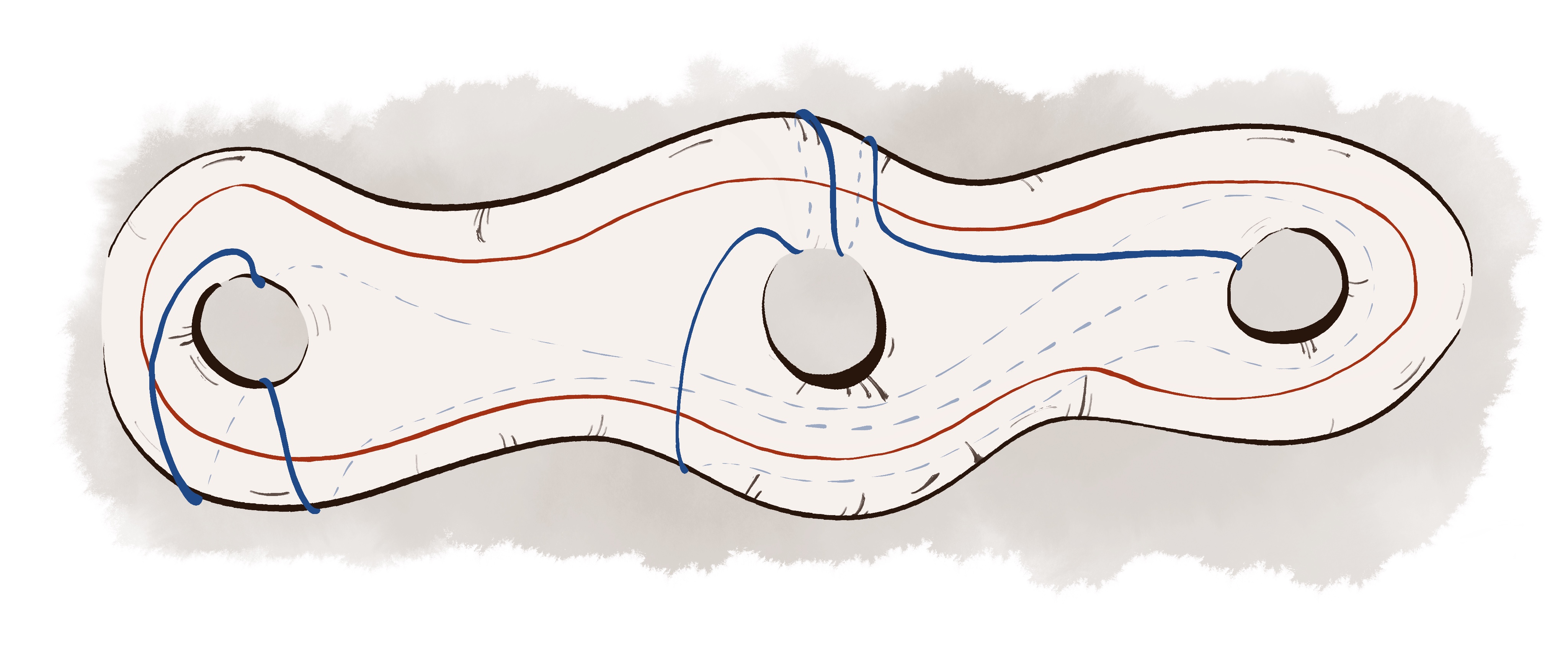}
    \caption{The complement of the red curve and the blue curve is a disk; hence, they form a filling pair.}
    \label{fig:filling pair}
\end{figure}

\subsection{Thickened multicurve}\label{subsec:thickenedmulticurve}

\begin{comment}
\begin{definition} [thickened multicurve]
\label{definition: Thicken system of curves}
A system of curves $L$, together with its regular neighbourhood, is called a \emph{thickened multicurve}, denoted by $T$. We call $L$ the \emph{core} of $T$. The disjoint components of $\partial T$ are called the \emph{boundary components} of $T$.
In this paper, we only discuss the case where $T$ is orientable.
\end{definition}
\end{comment}

\begin{definition} [Thickened multicurve]
\label{definition: Thicken system of curves}
Let $T$ be an \emph{orientable} surface with boundary and $L$ be a multicurve on $T$ such that there exists a deformation retraction map $r:T\rightarrow L$. Then we call $T$ a \emph{thickened multicurve} and $L$ the \emph{core} of $T$.
\end{definition}

In particular, a disjoint union of annuli is an example of a thickened multicurve whose core is a collection of simple closed curves. When the core is a (complete) $k$-system of curves, we call it a \emph{thickened (complete) k-system of curves}.

\begin{definition} [Plumbing]
\label{definition: Plumbing}
Let $T$ be a thickened multicurve and let $L$ be its core. 
Let $\alpha_1,\alpha_2$ be two disjoint open segments of $L$. Perform the surgery along $\alpha_1$ and $\alpha_2$ as illustrated in the following figure. We call this operation a \emph{plumbing} surgery. (see \cite{MR358813}, p.114 for a precise definition)

Every thickened multicurve appearing in this paper can be obtained from a disjoint union of annuli (thickened single curves) by a sequence of such plumbing surgeries.
\end{definition}

\begin{figure} [H]
    \centering
    \includegraphics[width=1\linewidth]{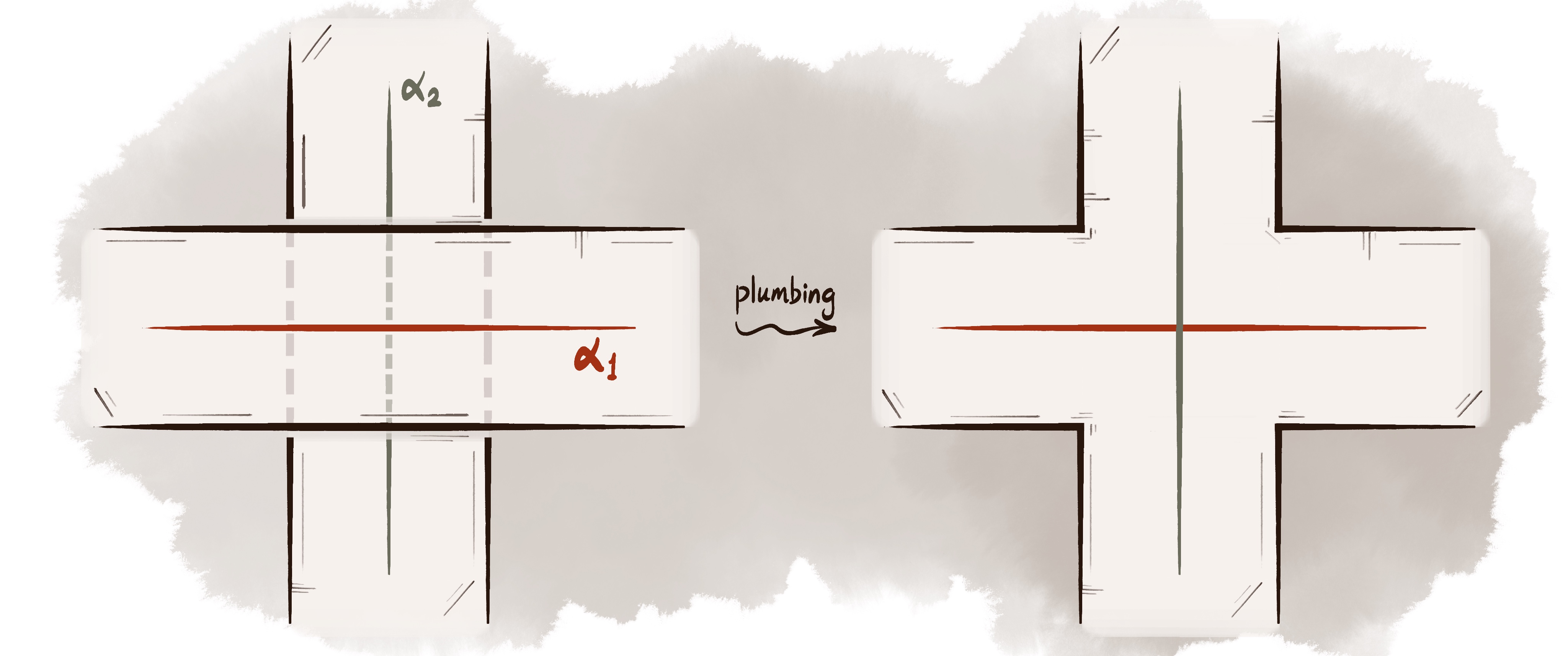}
    \caption{Plumbing.}
    \label{fig:Plumbing}
\end{figure}

\begin{definition}[Associated surfaces]
\label{definition:associated surfaces}
Let $T$ be a thickened multicurve. For each boundary component of $T$, we glue a disk to it and obtain a compact surface. We denote this new surface by $S_T$, and call it the \emph{associated surface of the thickened multicurve $T$}.
\end{definition}

\begin{proposition}
\label{proposition: genus of surface for TSL}
Suppose a multicurve $L$ has a total of $i$ intersection points and the thickened system $T$ of $L$ has $b$ boundary components. 
Then $L$ fills $S_T$, and the genus of $S_T$ is 
\begin{align}
    g = \frac{i - b + 2}{2}. \notag
\end{align}
\end{proposition}

\begin{proof}
By \cref{definition:associated surfaces}, if we cut the surface along the given multicurve, the resulting components are disks; hence the multicurve is filling. The induced cell decomposition has $i$ vertices, $2i$ edges, and $b$ faces. By the Euler characteristic formula,
\begin{align}
    2-2g=\chi(S)=i-2i+b. \notag
\end{align}
Solving for $g$ gives $g=\frac{i-b+2}{2}$, as claimed.
\end{proof}

\begin{corollary}
\label{corollary: both odd or even}
    The number of intersections $i$ and boundary components $b$ are both odd or both even. In particular, when $i$ is even, $b\ge 2$.
\end{corollary}

\begin{theorem}\label{thm:lowerbound}
If $S_g$ admits a filling $k$-system of $n$ curves, then $g \leq g_{k,n}$, where $g_{k,n}$ is defined in \eqref{eq:gkn}. The minimal number of curves in a filling $k$-system on $S_g$ is at least $n_0 = n_0(k,g)$ defined in \eqref{equ:n0}.
\end{theorem}

\begin{proof}
Suppose we have a filling $k$-system on $S_g$ consisting of $n$ curves that pairwise intersect at most $k$ times. 
Then the total number of intersection points is at most $\frac{1}{2}kn(n-1)$. 
Let the number of boundary components be $b\ge 1$. 
By \cref{proposition: genus of surface for TSL}, we have
\begin{equation}
g=\frac{1}{4}kn(n-1)-\frac{b}{2}+1
\le
\left\lfloor \frac{1}{4}kn(n-1)+\frac{1}{2}\right\rfloor .
\end{equation}
Equality holds if and only if $b=1$ or $2$. 
Therefore we obtain
\begin{equation}
n \ge \left\lceil \frac{1}{2}\left(1+\sqrt{\frac{16g-8}{k}+1}\right) \right\rceil ,
\end{equation}
where $\lfloor x\rfloor$ (resp. $\lceil x\rceil$) denotes the greatest integer less than or equal to $x$ (resp. the smallest integer greater than or equal to $x$).
\end{proof}

In the next few sections, we prove that this lower bound is optimal. Specifically, we construct a thickened multicurve $T$ whose number of boundary components is either 1 or 2. This shows that the upper bound on the genus is achieved, thereby establishing the optimality of the lower bound.

\section{Construction of $k$-system on $S_{g_{k,n}}$}\label{sec:kodd}

In this section, we give a construction of a filling $k$-system of $n$ curves on $S_{g_{k,n}}$, which shows that the upper bound in Theorem~\ref{thm:lowerbound} is optimal. This completes the proof of Theorem~\ref{thm:main1}. 

We first construct a thickened $k$-system consisting of $n$ curves, called the \textit{stairs}. Then, by modifying the nature of the intersection points, specifically through the transformations later defined as Reidemeister I and Reidemeister II, we eventually obtain a thickened $k$-system with only one or two boundary components.

Note that the constructions for odd $k$ and even $k$ are quite different; moreover, in each case the construction also depends significantly on the residue of $n$ modulo $4$, and when $n$ is small ($n = 3, 4, 5$) we need some special constructions. To make the construction easier to follow, we give a brief introduction to the structure of the section below. 

In Section~\ref{sec:def_stairs} we define the \emph{stairs} as a thickened $k$-system in Definition~\ref{def:stairs}, and define \emph{Reidemeister moves} in Definitions~\ref{def:Reidemeister1} and \ref{def:Reidemeister2} which are used to carry out the construction. In Section~\ref{sec:construction:ngeq6} we give the construction for $n \geq 6$. We discuss the case $k = 1$ in Section~\ref{sec:construction:ngeq6_1}, and the case $k \geq 2$ in Section~\ref{sec:construction:ngeq6_2}, where we distinguish four cases based on the residue of $n$ modulo $4$. In particular, the constructions for $k \equiv 2 \pmod{4}$ differ significantly between odd and even $k$. In Section~\ref{sec:construction:nleq5} we give the constructions for $n = 3, 4, 5$ respectively. For $n = 3$ and odd $k\geq3$, a special operation beyond Reidemeister moves is required.

\subsection{Definitions}\label{sec:def_stairs}

\begin{definition}[Stairs]\label{def:stairs}
    A \emph{stairs} $St(n,k)$ is defined as the following staircase-shaped thickened $k$-system. 
    Place $n$ arcs $\delta_1,\dots,\delta_n$ in the plane such that for each $1 \le j \le n$, the arc $\delta_j$ starts at the point $\bigl(\frac{j}{n}, 0\bigr)$, then follows  a path of upward and rightward straight line segments passing successively through the points  
\[
\Bigl(\frac{j}{n},\frac{j}{n}\Bigr),\;
\Bigl(\frac{j}{n}+1,\frac{j}{n}\Bigr),\;
\Bigl(\frac{j}{n}+1,\frac{j}{n}+1\Bigr),\dots,\;
\Bigl(\frac{j}{n}+\frac{k-1}{2},\frac{j}{n}+\frac{k-1}{2}\Bigr),
\] 
    and ends at $\bigl(\frac{1}{n}+\frac{k+1}{2},\frac{j}{n}+\frac{k-1}{2}\bigr)$. If $k$ is even, it passes successively through the points 
\[
\Bigl(\frac{j}{n},\frac{j}{n}\Bigr),\;
\Bigl(\frac{j}{n}+1,\frac{j}{n}\Bigr),\;
\Bigl(\frac{j}{n}+1,\frac{j}{n}+1\Bigr),\dots,\;
\Bigl(\frac{j}{n}+\frac{k}{2},\frac{j}{n}+\frac{k-2}{2}\Bigr),
\] 
    and ends at $\bigl(\frac{j}{n}+\frac{k}{2},\frac{1}{n}+\frac{k}{2}\bigr)$. Then every pair of distinct arcs $\delta_i,\delta_j$ intersect exactly $k$ times.

    %\large \textbf {To Xiao Chen, please draw a case when $k$ is odd and $k$ is even instead of 1 and 3.}

    \normalsize

    Taking a tubular neighbourhood of the whole configuration and then cyclically identifying the endpoints of each arc $\delta_j$, we simultaneously glue the corresponding ends of the tubular neighbourhood. This gluing is performed in the unique way that makes the resulting thickened multicurve orientable. The resulting object is the thickened multicurve $St(n,k)$.
\end{definition}

\begin{figure}[H]
    \centering
    \includegraphics[width=1\linewidth]{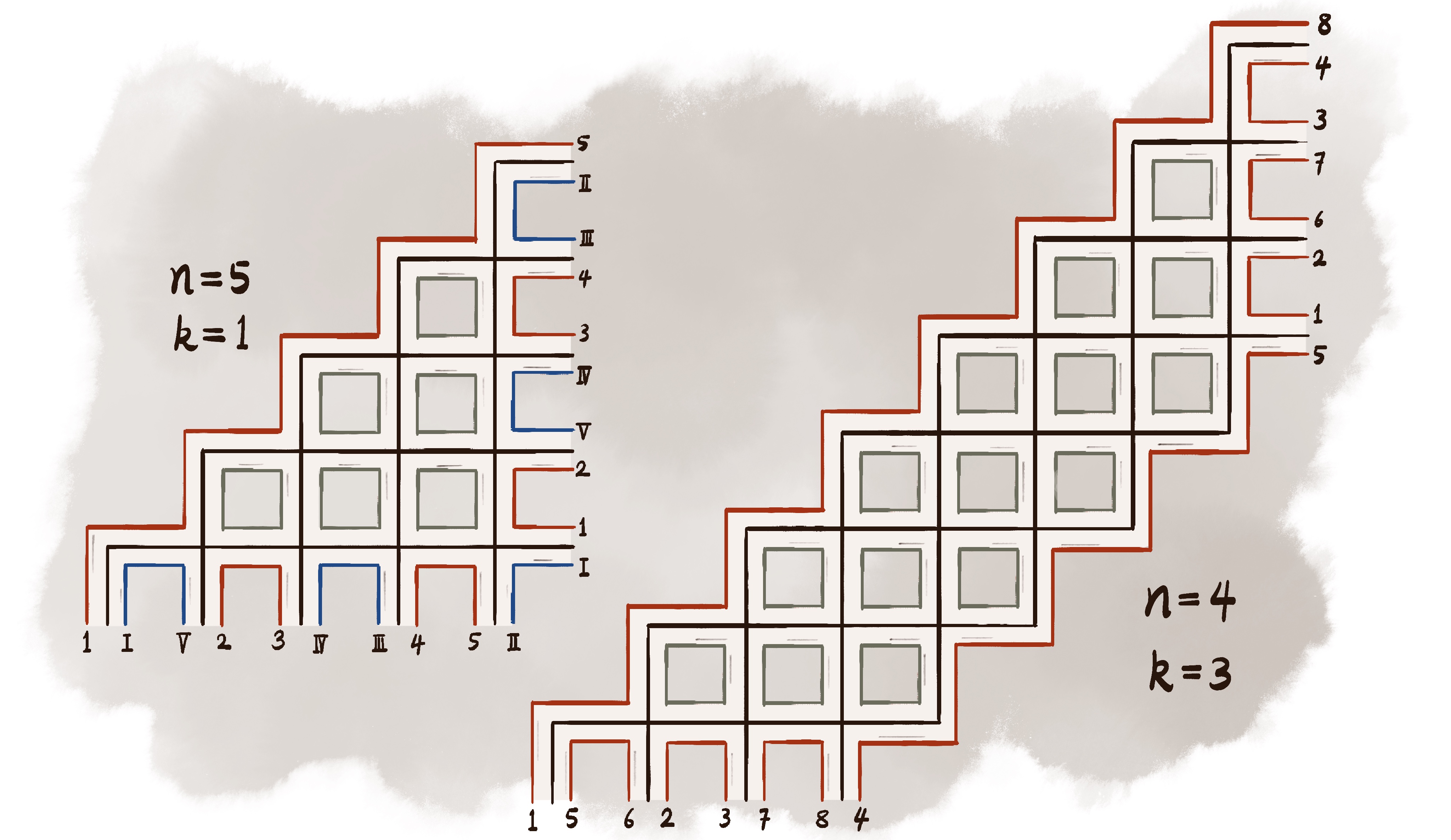}
    \caption{Examples of stairs when $k$ is odd.}
    \label{fig:k1n5k3n4_1}
\end{figure}

\begin{figure}[H]
    \centering
    \includegraphics[width=1\linewidth]{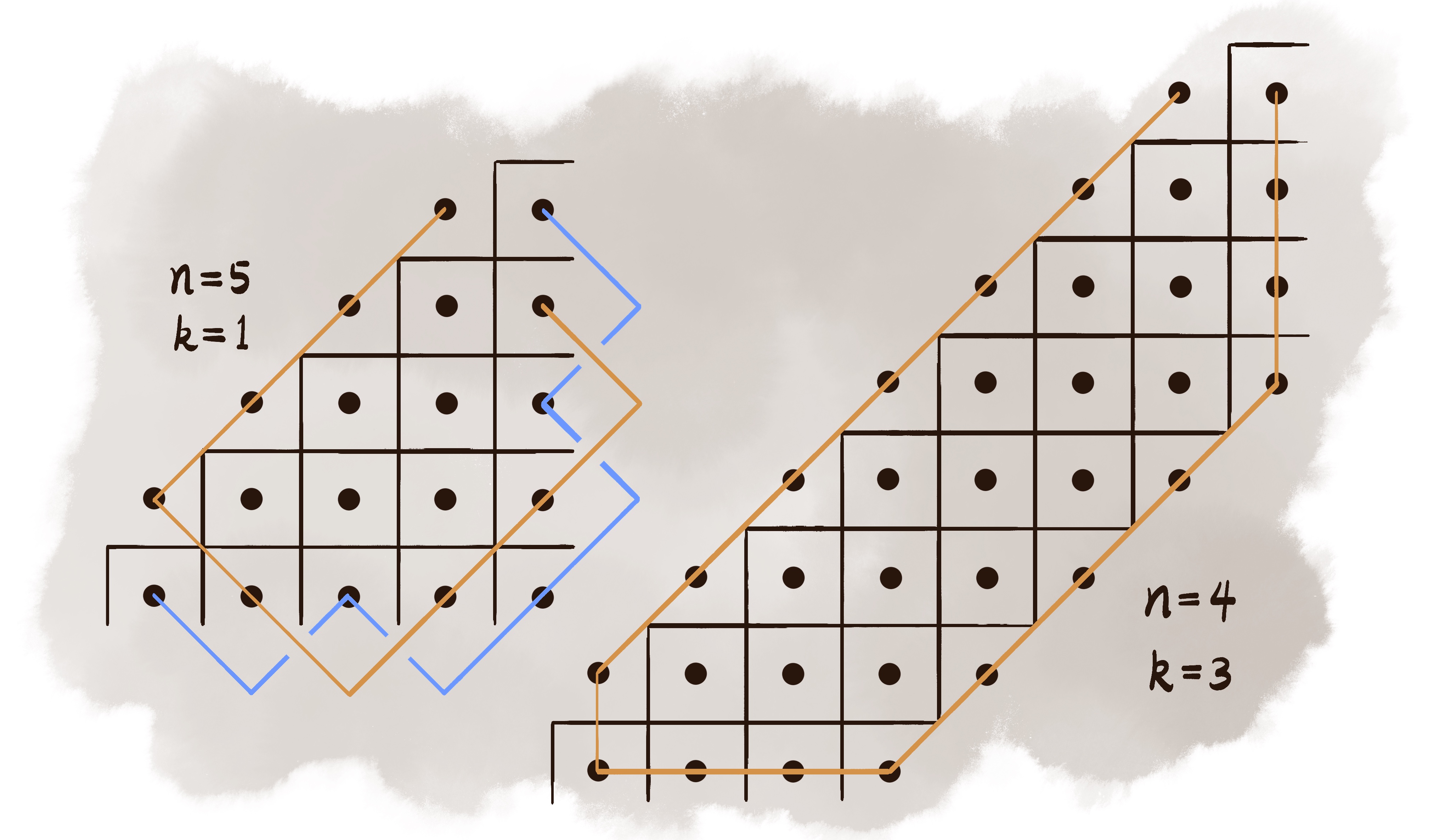}
    \caption{Examples for odd \(k\). In the boundary of the thickened system, the outer parts are treated as vertices; two vertices are joined if the corresponding parts lie in the same boundary component. When both \(n\) and \(k\) are odd, the outer part consists of two components (coloured yellow and blue); when \(n\) is even and \(k\) is odd, the outer part consists of a single component (coloured yellow).}
    \label{fig:k1n5k3n4_2}
\end{figure}

\begin{figure}[H]
    \centering
    \includegraphics[width=1\linewidth]{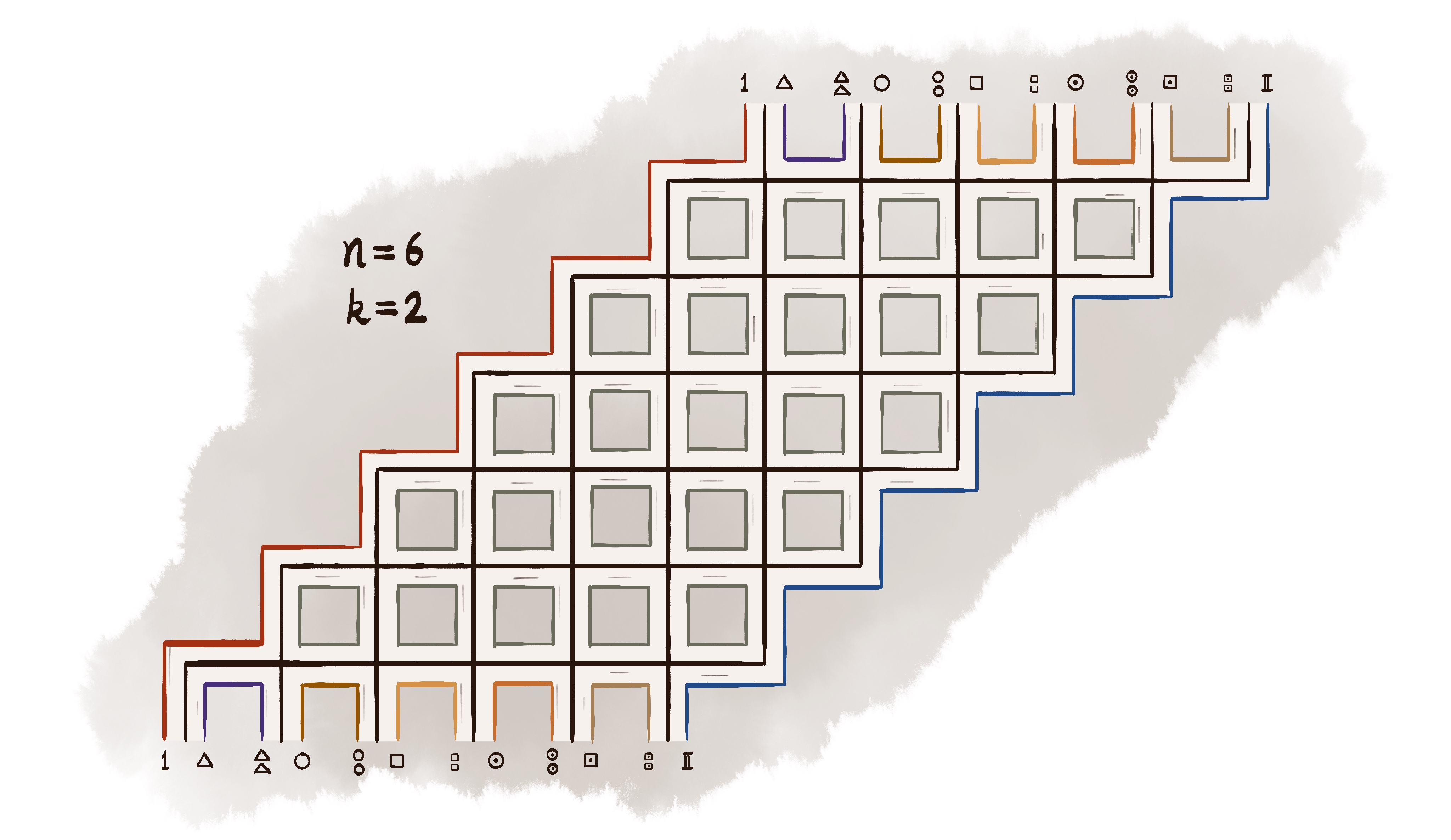}
    \caption{An example of a stair when $k$ is even.}
    \label{fig:k2n6_1}
\end{figure}

\begin{figure}[H]
    \centering
    \includegraphics[width=1\linewidth]{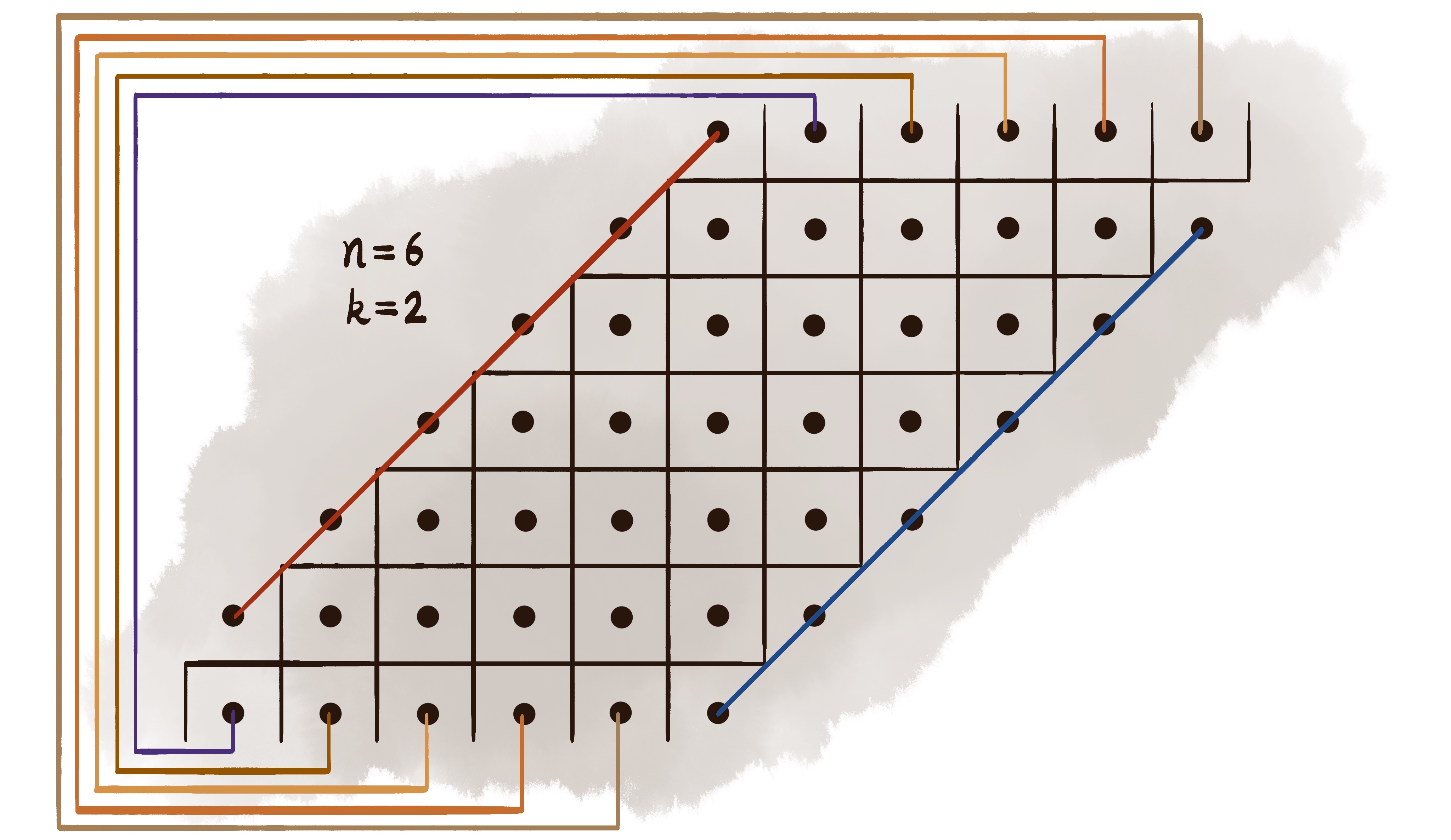}
    \caption{Examples for even $k$. In the boundary of the thickened system, the outer part contains $n+1$ components: a left component (coloured red), a right component (coloured blue), and the remaining $n-1$ components, where the parts above and below are paired in order.}
    \label{fig:k2n6_2}
\end{figure}

When $k$ is odd, the two examples shown above are $St(5,1)$ and $St(4,3)$ as shown in Figure \ref{fig:k1n5k3n4_2}. For each $i = 1,\dots,n$, $A_i$ is glued to $A_i'$ and $B_i$ is glued to $B_i'$. When $k$ is odd and $n$ is odd, the outer boundary of $St(n,k)$ in the above figure consists of two connected components, coloured red and blue; when $k$ is odd and $n$ is even, the outer boundary has only one connected component.

When $k$ is even, the example shown above is $St(6,2)$ as shown in Figure \ref{fig:k2n6_2}. There are two boundary components on the left and right, and each pair of cells on the top and bottom is a distinct boundary component, giving $n+1$ boundary components in total.

\begin{definition}[Reidemeister I move]\label{def:Reidemeister1}
    For each intersection point, we can reverse the orientation of the thickened multicurve at that point using the method shown on the left of Figure~\ref{fig:Surgeries.}, yielding a new thickened multicurve. This operation is called a \emph{Reidemeister I move}. Note that around an intersection point there are four components; after performing this orientation reversal, the diagonal components are paired together, and each pair is connected into a single component.
\end{definition}

\begin{definition}[Reidemeister II move]\label{def:Reidemeister2}
    For each pair of adjacent intersection points, we can exchange the positions of the two intersections on the thickened multicurve as shown on the right of Figure~\ref{fig:Surgeries.}, yielding a new thickened multicurve. This operation is called a \emph{Reidemeister II move}. Note that around the two intersection points there are six components—three above and three below. After performing this surgery,  the three components above are connected into a single component, and similarly the three components below are connected into a single component.
\end{definition}

    \begin{figure} [H]
        \centering
        \includegraphics[width=1\linewidth]{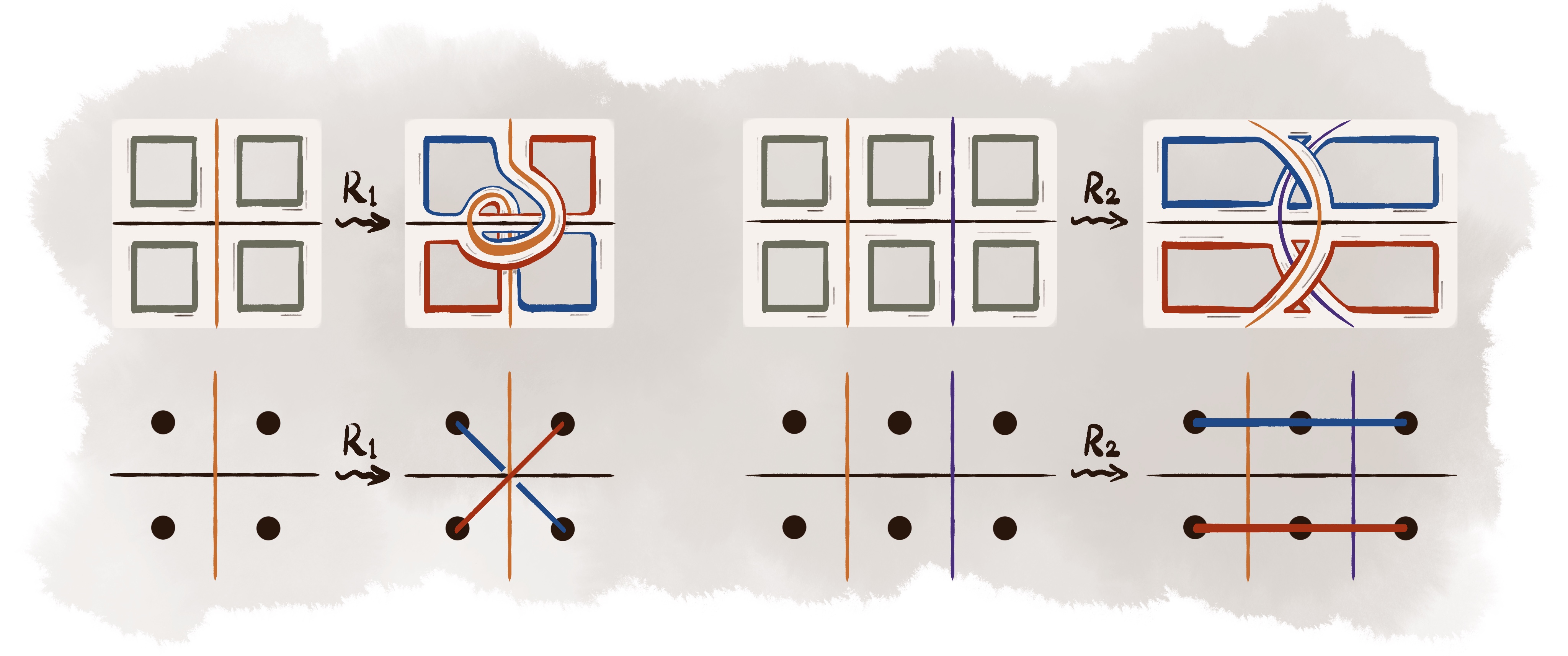}
        \caption{The Reidemeister moves.}
        \label{fig:Surgeries.}
    \end{figure}

In the next two subsections, starting from the stairs, we apply Reidemeister I and Reidemeister II moves at intersection points to reduce the number of boundary components of the thickened multicurve to one or two.

In a stairs, at an intersection point there are four surrounding components: upper‑left, upper‑right, lower‑left, and lower‑right. Performing a Reidemeister~I move at that point connects the upper‑left component with the lower‑right component, and simultaneously connects the upper‑right component with the lower‑left component. If the upper‑left and lower‑right belong to different components, the upper‑right and lower‑left also belong to different components, and these two pairs are not the same pair of components, then the Reidemeister~I move reduces the number of boundary components by \(2\) and increases the genus of the surface by \(1\). In the left part of Figure~\ref{fig:Surgeries.}, a cross mark indicates a Reidemeister~I move.

If around two adjacent intersections there are six components such that the three components on one side are pairwise distinct, the three components on the other side are also pairwise distinct, and the two triples share at most one component in common, then performing a Reidemeister~II move reduces the number of boundary components by \(4\) and increases the genus of the surface by \(2\). The move connects the three components on each side into a single component. In the right part of Figure~\ref{fig:Surgeries.}, two parallel horizontal or vertical lines indicate a Reidemeister~II move. In the remainder of the paper, we perform Reidemeister~I and II moves only when all of these conditions are satisfied, and we call them {\em legal moves}.

In the following sections we will perform multiple Reidemeister I and II moves on the stairs, and we will call a sequence of such moves a {\em scheme of Reidemeister moves}. We note that the order of the moves does not matter as long as we always connect different components.

%We note that if the Reidemeister I move on a intersection point, reduces the number of boundary components by $2$, and the Reidemeister II move reduces the number by $4$. 

\subsection{The attaching graph argument}

To determine the number of boundary components after Reidemeister moves and to avoid connecting the same components, we introduce an {\em attaching graph} $G$, which is motivated by \cite{CM}. 

The vertices of $G$ are distinct boundary components of the stairs $St(n,k)$. When $k$ is even, as shown in Figure \ref{fig:k1n5k3n4_2}, there are one or two outer boundary components so there are one or two vertices representing the outer part (which is different from the figure) and each inner square is a vertex. When $k$ is odd, as shown in Figure \ref{fig:k2n6_2}, there are $n$ outer boundary components so there are $n$ outer vertices (which is also different from the figure) and each inner square is a vertex. Given a scheme of Reidemeister moves, two vertices share an edge if they are connected via a Reidemeister move. We note that a boundary component cannot be connected to itself, and connecting the same two components more than once is an illegal move, so $G$ is a simple graph.

If $G$ can be realised by a scheme of Reidemeister moves, we call $G$ a {\em possible graph}. We note that all possible graphs contain an even number of edges.

\begin{theorem}
 \label{Theorem: tree}
        Given the stairs $St(n,k)$ and a scheme of Reidemeister moves, the resulting surface will achieve the maximum possible genus if and only if the associated attaching graph is possible and is a connected tree or a forest with two connected components.
\end{theorem}

\begin{proof}
    First we show that if the surface achieves the maximal genus, then the graph is a forest with one or two connected components. Suppose that there is a loop in $G$; then after a sequence of moves there will be a Reidemeister move connecting the same two boundary components, which is illegal. Then we notice that if we connect two vertices they will make a single boundary component after the move, so the number of connected components in $G$ is the same as the number of boundary components of the surface. According to Theorem \ref{thm:lowerbound}, the optimal number is 1 or 2.

    For the converse we notice that if $G$ is possible and is a forest then every Reidemeister move connects distinct boundary components so the moves are all legal. And we notice that Reidemeister moves always change the number of boundary components by an even number,  when the number of connected components in $G$ reaches 1 or 2, the genus will reach its maximum for $St(n,k)$.
\end{proof}

So our goal is to reduce the boundary components to form a forest with one or two connected components. In the following sections, we will give the constructions. It is easy to verify that the attaching graphs satisfy Theorem \ref{Theorem: tree} and we leave this to the reader.

\subsection{Construction when $n\geq6$}\label{sec:construction:ngeq6}

\begin{definition}[Loco and carriage operation]\label{def:basicoperation}
    For a four‑step stair segment, we define a \emph{loco operation} as performing one Reidemeister I move followed by two Reidemeister II moves as illustrated in the red part of Figure \ref{fig:Loco}. For the part to the right of the loco part, we perform consecutive Reidemeister I moves (indicated by crosses) as shown in the blue part of Figure \ref{fig:Loco}, we call this \emph{carriage operation}. These two operations connect the components inside the four-step segment to the outside components.
\end{definition}

        \begin{figure} [H]
        \centering
        \includegraphics[width=0.8\linewidth]{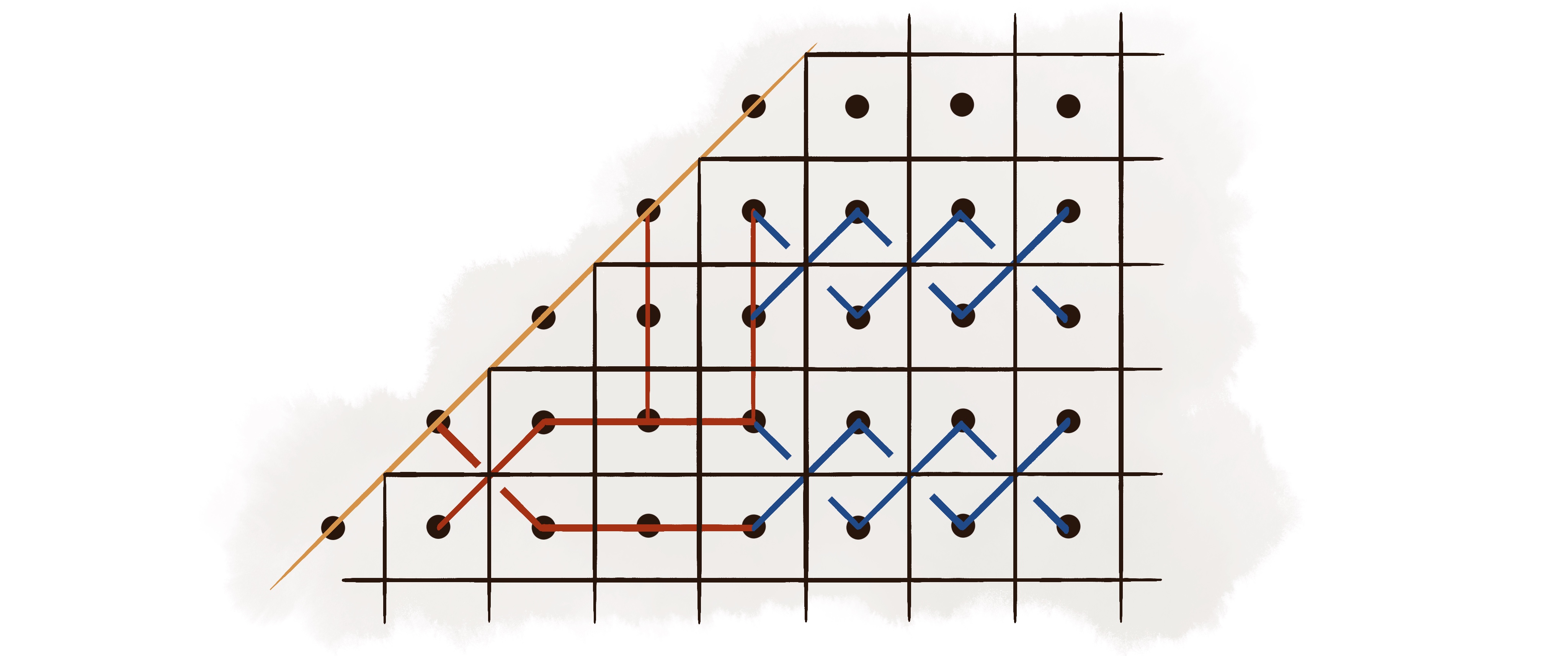}
        \caption{The loco and carriage operation.}
        \label{fig:Loco}
    \end{figure}

\subsubsection{The case when $k=1$}\label{sec:construction:ngeq6_1}

For \(k=1\), we group the rows from bottom to top in blocks of four. In each four-row block, the leftmost four-step stair segment is constructed using the loco and carriage operation defined in \cref{def:basicoperation}.

At the top, at most three rows of small squares remain untreated. Depending on the residue of \(n\) modulo \(4\), four cases are distinguished. In each case, applying suitable Reidemeister I and Reidemeister II moves reduces the number of boundary components to either one or two, see Figure \ref{fig:k1n01} and Figure \ref{fig:k1n23}.

Specifically, after performing the above operations on every four-row block, we have the following results(which define \emph{Construction I}):

\begin{enumerate}
    \item If \(n \equiv 0 \pmod{4}\), there remain three small squares at the top. Applying a single Reidemeister~I move connects two of these squares to the outside component, leaving only one small square and the outside component—hence a total of two boundary components.
    \item If \(n \equiv 1 \pmod{4}\), there remain six small squares at the top. Applying a single Reidemeister~I move and a single Reidemeister~II move connects five of these squares to the outside component, leaving only one small square and the outside component—hence a total of two boundary components.
    \item If \(n \equiv 2 \pmod{4}\), then after performing the above operations on every four-row block, the thickened multicurve already has exactly one boundary component, and no additional moves are required.
    \item If \(n \equiv 3 \pmod{4}\), then after performing the operations on the lower \(\frac{n-7}{4}\) four-row blocks, there remains a five-step stair segment consisting of 15 small squares at the top. Applying four Reidemeister~I moves and two Reidemeister~II moves reduces the thickened multicurve to a single boundary component. Note that the rightmost Reidemeister~I move joins the two outer boundary components of the stairs.
\end{enumerate}

    \begin{figure} [H]
        \centering
        \includegraphics[width=1\linewidth]{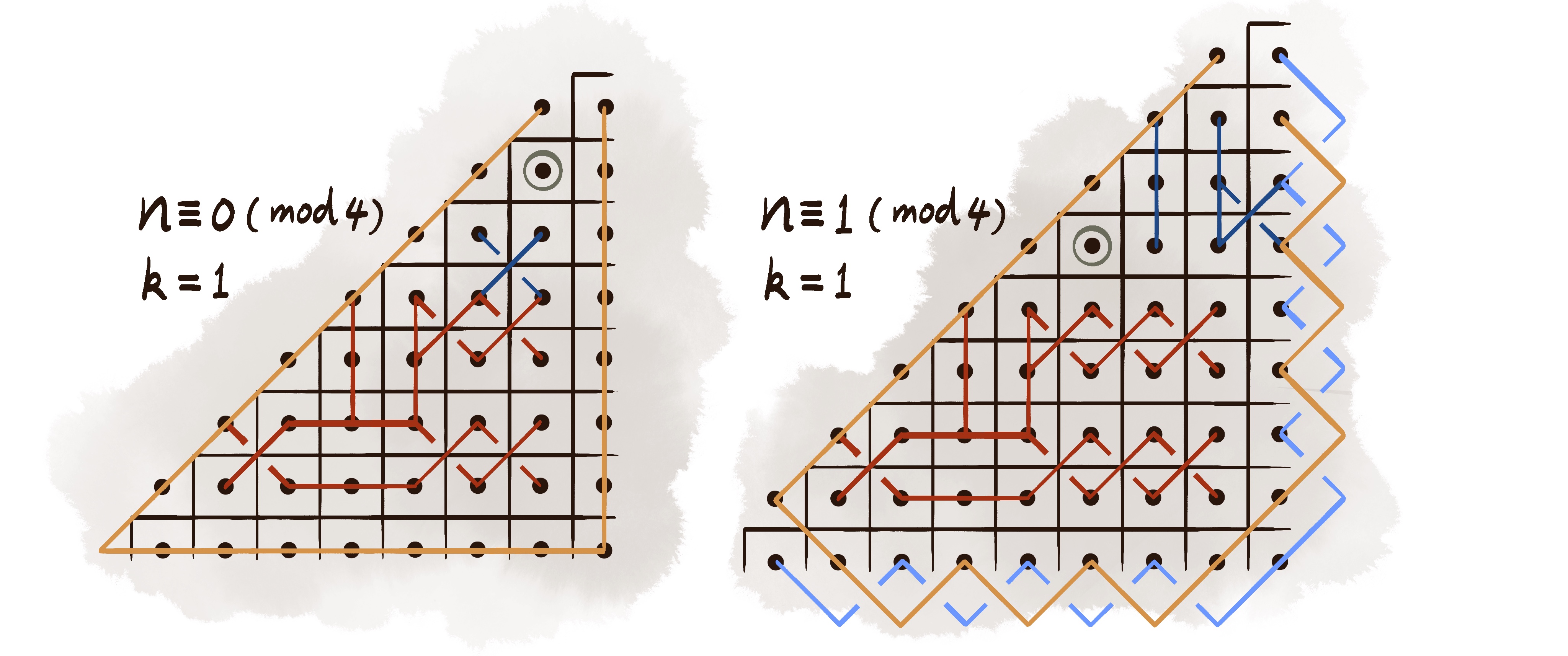}
        \caption{$k=1$, $n \equiv 0 \pmod{4}$ and $k=1$,  $n \equiv 1 \pmod{4}$.}
        \label{fig:k1n01}
    \end{figure}

    \begin{figure} [H]
        \centering
        \includegraphics[width=1\linewidth]{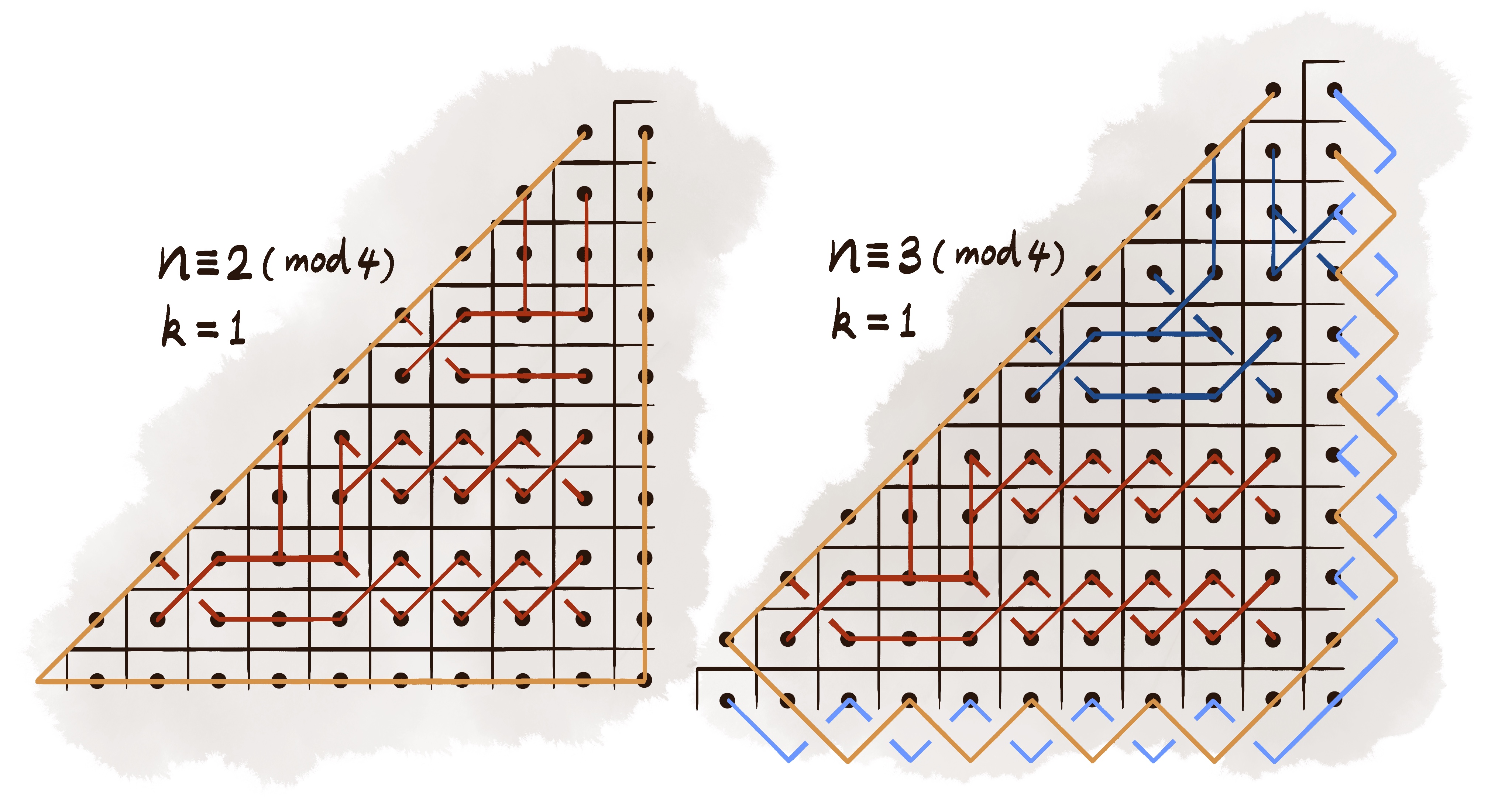}
        \caption{$k=1$, $n \equiv 2 \pmod{4}$ and $k=1$,  $n \equiv 3 \pmod{4}$.}
        \label{fig:k1n23}
    \end{figure}

\subsubsection{The case when $k\geq2$}\label{sec:construction:ngeq6_2}

When $k \geq 3$ is odd, the construction consists of one Construction~I (as defined for $k = 1$) followed below by $\frac{k-1}{2}$ copies of \emph{Construction~II} on the basic block $St(n,2)$ (defined below). When $k\geq2$ is even, we use $\frac{k}{2}$ copies of Construction~II on the basic block $St(n,2)$ (in the case $k \equiv 2 \pmod{4}$ a more specific construction is required). In either case, we only need to focus on Construction~II.

\begin{enumerate}
    \item $n \equiv 1 \pmod{4}$

\begin{comment}
    When \(k=3\), we group the small squares of \(St(n,3)\) into four-row blocks, obtaining \(\frac{n-1}{2}\) such blocks.  
    The lower \(\frac{n-1}{4}\) four-row blocks are parallelograms. In each of them, the left and right four-step stair segments are handled using the \emph{loco operation} defined in \cref{def:basicoperation}, while the middle part is processed by the \emph{carriage operation} introduced at the beginning of Section~4.2.1.  
    For the upper \(\frac{n-1}{4}\) four-row blocks, we adopt a construction analogous to the case \(k=1\): the left four-step segment uses the loco operation, and the remaining part on the right uses the carriage operation.  
    After these procedures, the thickened multicurve has its number of boundary components reduced to two. This is defined as \emph{Construction I}.
\end{comment}

    When \(k=2\), consider the parallelogram-shaped grid \(St(n,2)\), the grid is called the \emph{basic block}. Divide all small squares into \(\frac{n-1}{4}\) four-row blocks.

    \begin{figure} [H]
        \centering
        \includegraphics[width=0.9\linewidth]{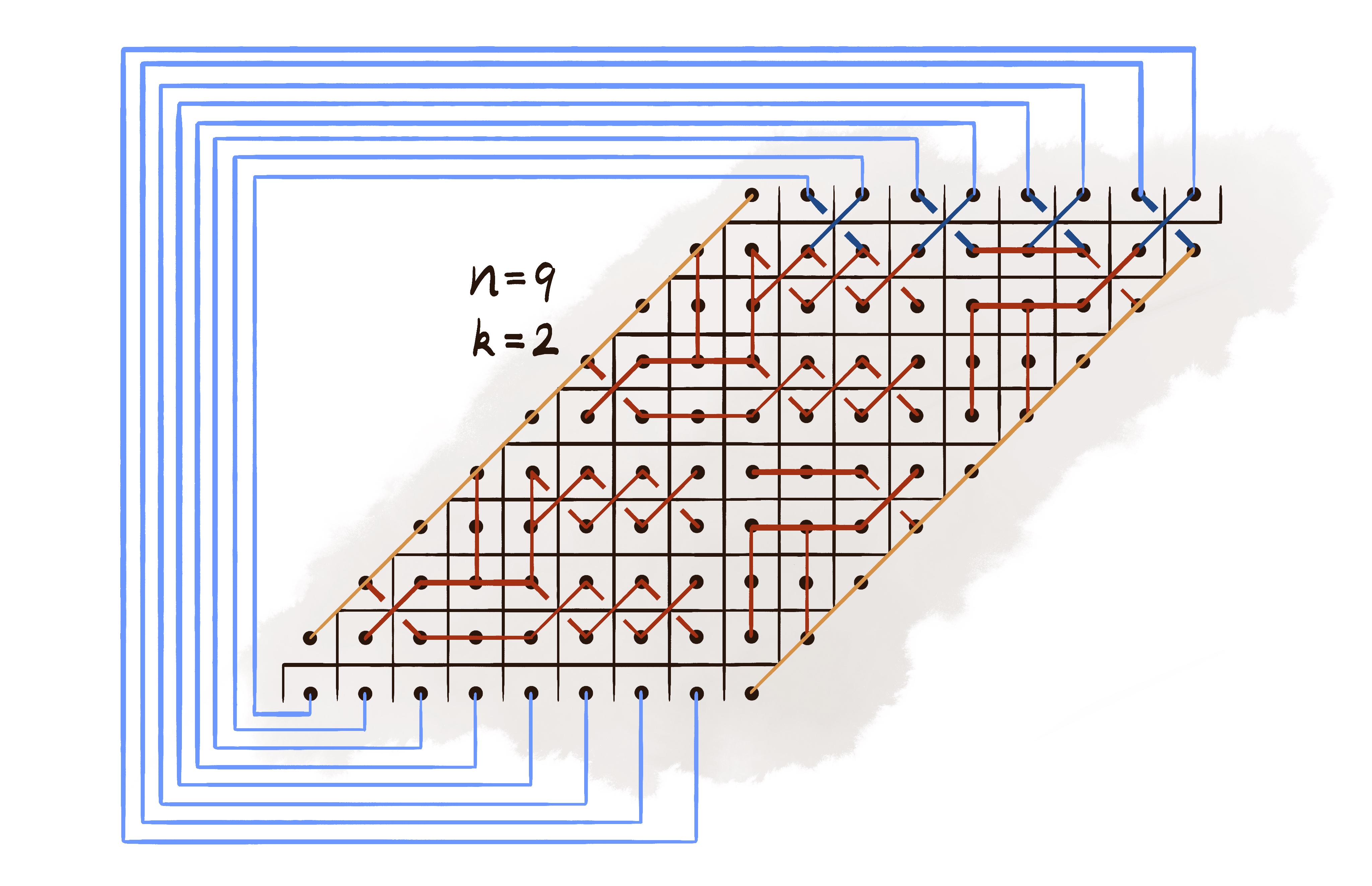}
        \caption{$k=2$, $n = 9$.}
        \label{fig:k2n9}
    \end{figure}

    In each block, the leftmost and rightmost four‑step stair segments are treated with the \emph{loco operation}, and the middle part is handled by the \emph{carriage operation}. The loco operations together with the carriage operations performed within a four-row block are called a \emph{basic construction}.

     On top of this, perform \(\frac{n-1}{2}\) Reidemeister~I moves between the uppermost row of squares and the row immediately above it; see Figure \ref{fig:k2n9}. We call this foundational construction \emph{Construction~II} ($n \equiv 1 \pmod{4}$).

    \begin{figure} [H]
        \centering
        \includegraphics[width=0.9\linewidth]{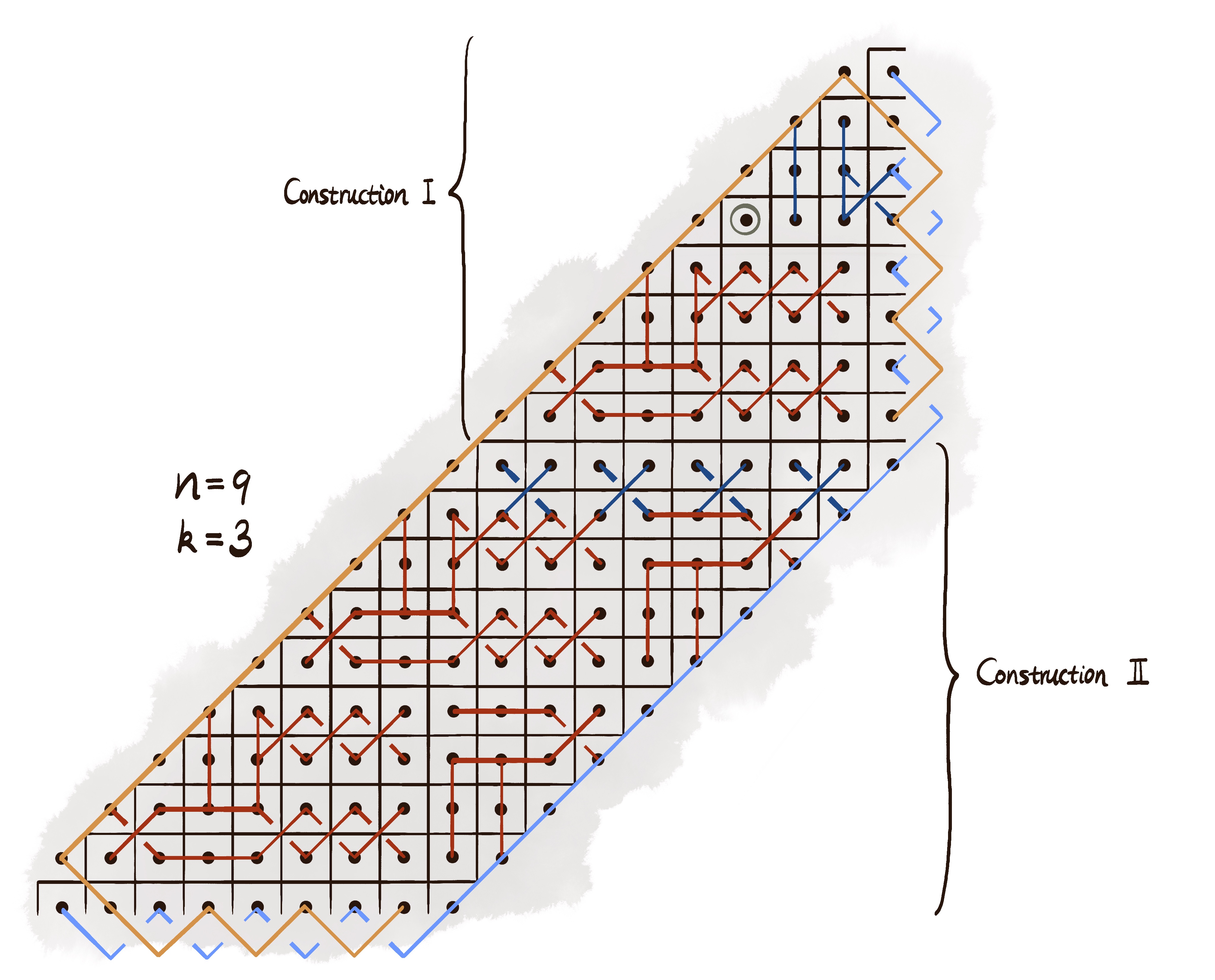}
        \caption{$k=3$, $n = 9$.}
        \label{fig:k3n9}
    \end{figure}

    When \(k\) is an odd integer with \(k \ge 3\), we assemble \(\frac{k-1}{2}\) copies of Construction~II sequentially below a Construction~I. This provides a procedure that, starting from \(St(n,k)\), applies a series of Reidemeister~I and Reidemeister~II moves and yields a thickened \(k\)-system with exactly two boundary components.

    The method of attaching Construction~II below the existing structure is the same in the remaining three cases, so we will only discuss the construction of Construction~II in what follows.

    \item $n \equiv 2 \pmod{4}$

    When \(k = 4\), consider the parallelogram-shaped grid \(St(n,4)\). Divide all small squares into \(\frac{3(n-2)}{4}\) four-row blocks; the row immediately above the top row of small squares also belongs to the uppermost four-row block. In each block, the leftmost and rightmost four-step stair segments are treated with the \emph{loco operation}, and the middle part is handled by the \emph{carriage operation}. We call this foundational construction \emph{Construction~II} when $n \equiv 2 \pmod{4}$. Note that this is different from the case when $n \equiv 0,1,3 \pmod{4}$, since we need to construct according to the residue of $k$ modulo $4$ on a case‑by‑case basis.

\begin{itemize}
    \item $k$ is odd. When \(k \equiv 1 \pmod{4}\), we assemble \(\frac{k-1}{4}\) copies of Construction~II sequentially below a Construction~I. This provides a procedure that, starting from \(St(n,k)\), applies a series of Reidemeister~I and Reidemeister~II moves and yields a thickened \(k\)-system with exactly two boundary components.

    \begin{figure} [H]
        \centering
        \includegraphics[width=1\linewidth]{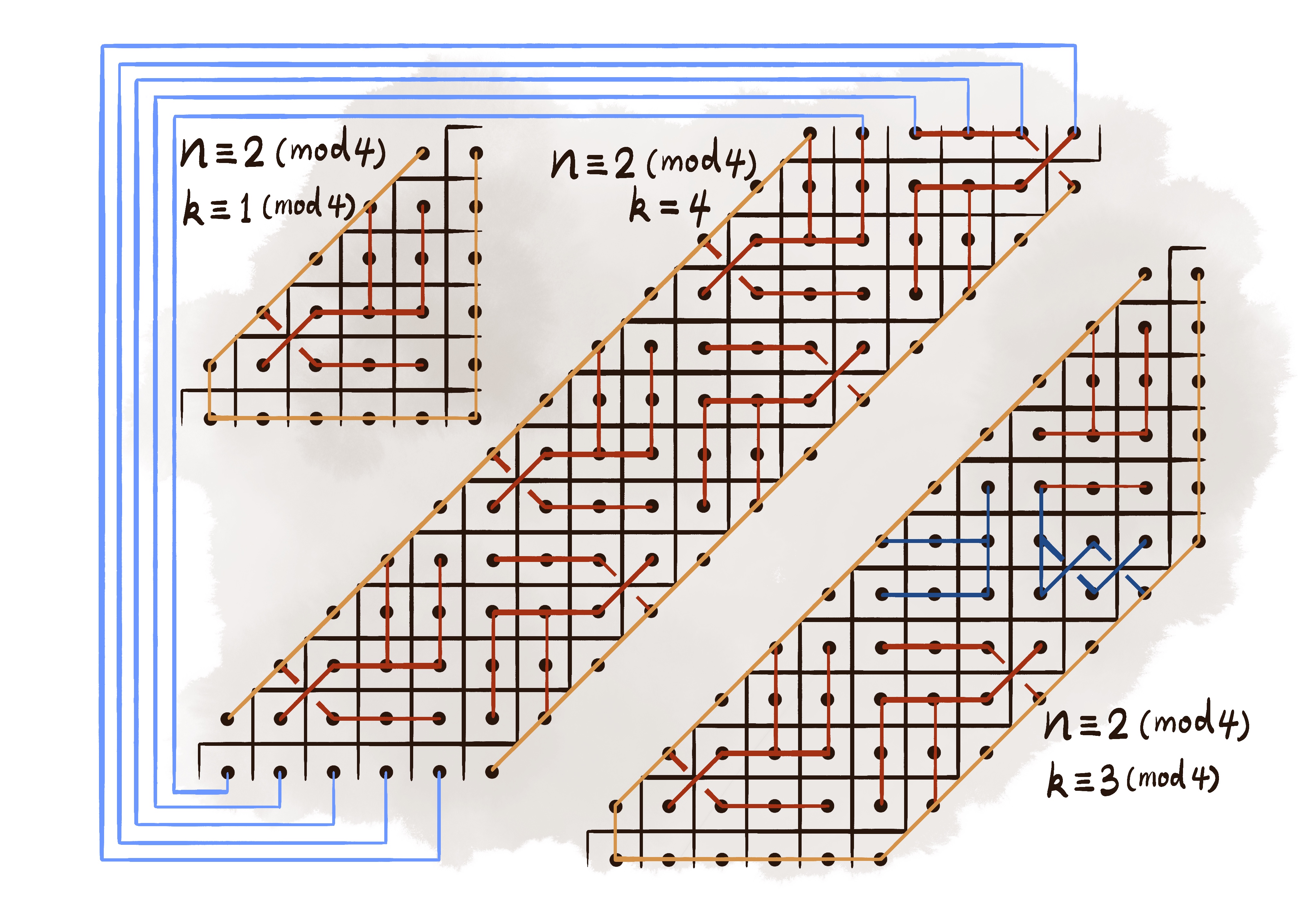}
        \caption{$n \equiv 2 \pmod{4}$.}
        \label{fig:n2}
    \end{figure}

    \item When \(k \equiv 3 \pmod{4}\), we assemble \(\frac{k-3}{4}\) copies of Construction~II sequentially below a Construction~I. Figure \ref{fig:k3n14} represents Construction I for $n=14$ and $k=3$. 
    
    Here “Construction~I” refers to the following: for \(St(n,3)\), take its lower \(\frac{n-2}{4}\) four-row blocks. In each block, the leftmost and rightmost four-step stair segments are treated with the loco operation, and the middle part is handled by the carriage operation. In each of the uppermost \(\frac{n-6}{4}\) four-row blocks, the leftmost four-step stair segment is constructed using the loco operation, and the part to its right is constructed using the carriage operation.
    
    This leaves six rows in the middle. In the bottom two of these rows, apply Reidemeister~I moves and then perform one vertical Reidemeister~II move and one horizontal Reidemeister~II move on the left side(see the blue Reidemeister moves in Figure \ref{fig:k3n14}). In the upper four rows, perform two Reidemeister~II moves and then  apply the carriage operation to the right side of these four rows.

    \begin{figure} [H]
        \centering
        \includegraphics[width=1\linewidth]{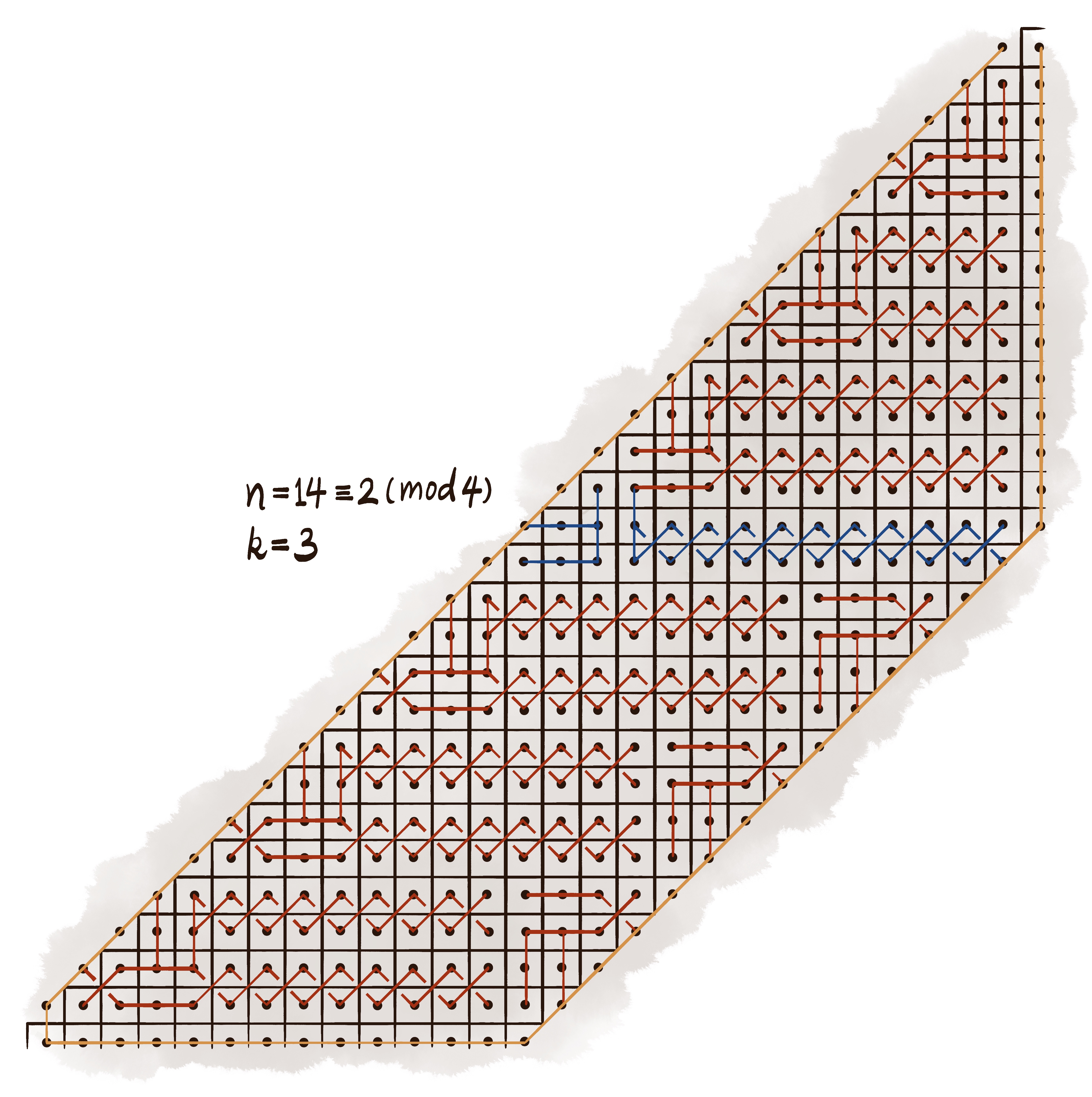}
        \caption{$k=3,n = 14 \equiv 2 \pmod{4}$.}
        \label{fig:k3n14}
    \end{figure}

     \item $k$ is even. When \(k \equiv 0 \pmod{4}\), we assemble \(\frac{k}{4}\) copies of Construction~II.
     
     When $k = 2$, we take the bottom \(\frac{n-2}{4}\) four-row blocks and perform the basic construction in each of them, and then apply Reidemeister~I moves to the top two rows. In this case, the squares of $St(n,2)$ are connected into two components. One component connects the leftmost and rightmost boundary components via the Reidemeister~I moves in the top two rows, while the other component lies entirely within the top two rows. See Figure \ref{fig:k2n10}.

    \begin{figure} [H]
        \centering
        \includegraphics[width=1\linewidth]{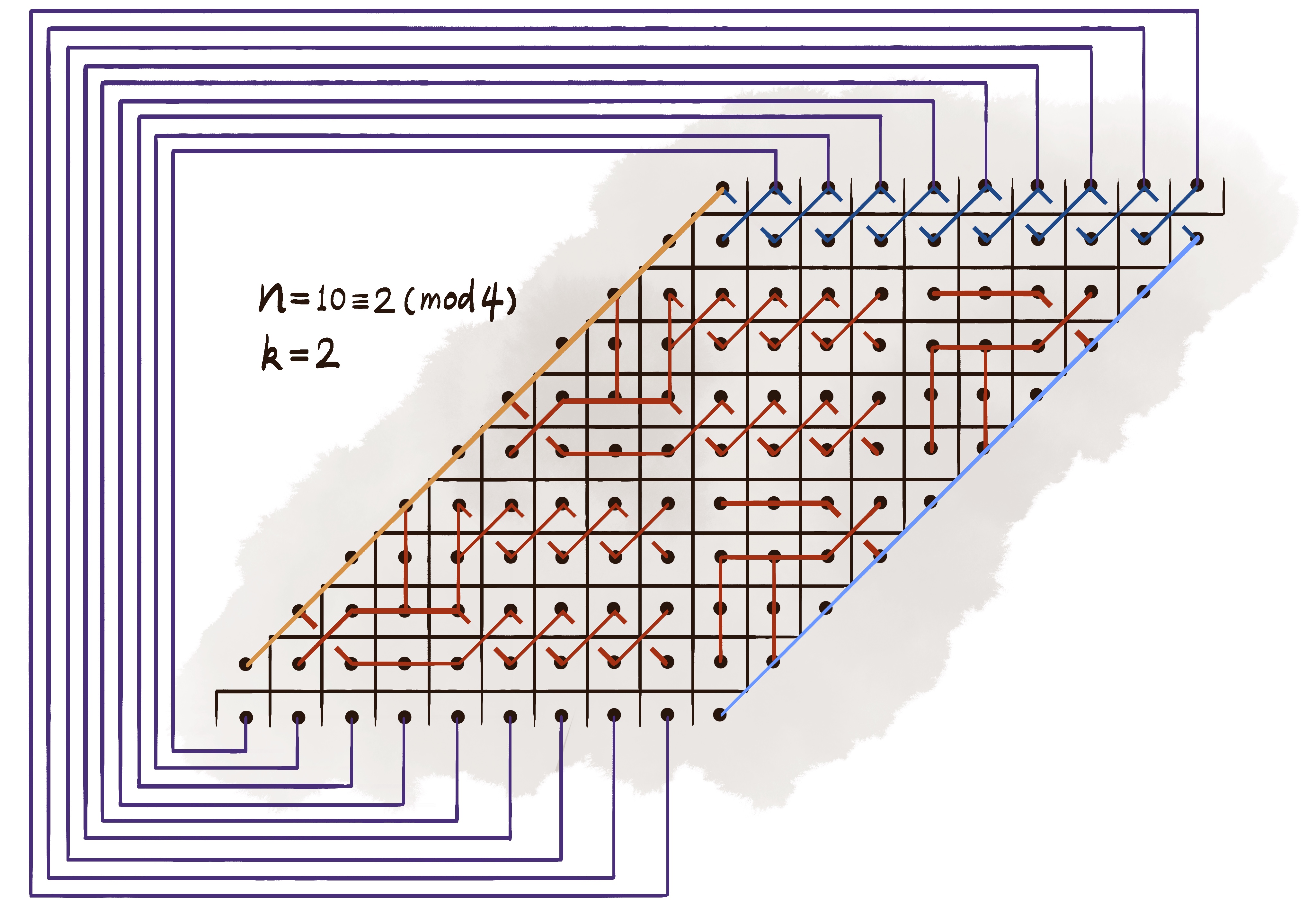}
        \caption{$k=2$, $n = 10 \equiv 0 \pmod{4}$.}
        \label{fig:k2n10}
    \end{figure}
     
     When \(k \equiv 2 \pmod{4}\), we assemble \(\frac{k-2}{4}\) copies of Construction~II at the bottom of the previous construction in $St(n,2)$.
\end{itemize}
    
    \item $n \equiv 3 \pmod{4}$

    When \(k = 2\), consider the parallelogram-shaped grid \(St(n,2)\). Divide the small squares into \(\frac{n-3}{4}\) four-row blocks at the bottom, leaving the top two rows of the small squares. In each lower block, the leftmost and rightmost four-step stair segments are treated with the \emph{loco operation}, and the middle part is handled by the \emph{carriage operation}. On the top two rows, perform $n-1$ Reidemeister~I moves between the uppermost two rows of squares, and connect the squares in the row immediately above them to the uppermost row by $\frac{n-1}{2}$ Reidemeister~I moves, as illustrated in Figure~\ref{fig:k1n01}. We call this foundational construction \emph{Construction~II} ($n \equiv 3 \pmod{4}$).

    \item $n \equiv 0 \pmod{4}$

    When \(k = 2\), consider the parallelogram-shaped grid \(St(n,2)\). Divide all small squares into \(\frac{n}{4}\) four-row blocks; note that the row immediately above the top row of small squares is also included in the uppermost four-row block. In each block, the leftmost and rightmost four-step stair segments are treated with the \emph{loco operation}, and the middle part is handled by the \emph{carriage operation}. We call this foundational construction \emph{Construction~II} when $n \equiv 0 \pmod{4}$.

    \begin{figure} [H]
        \centering
        \includegraphics[width=1\linewidth]{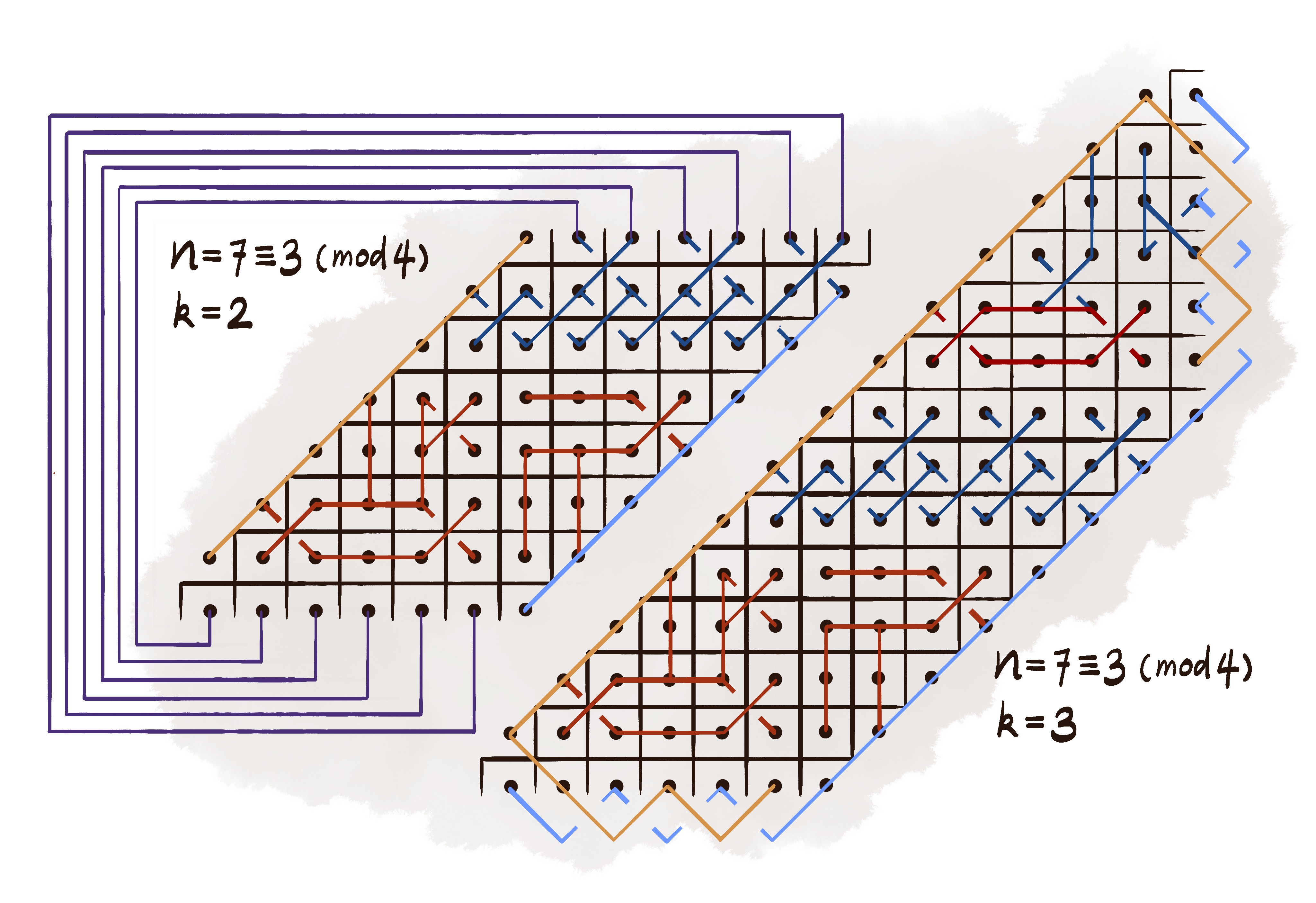}
        \caption{$k=2$, $n = 7 \equiv 3 \pmod{4}$ and $k=3$,  $n = 7 \equiv 3 \pmod{4}$.}
        \label{fig:k1n01}
    \end{figure}

    \begin{figure} [H]
        \centering
        \includegraphics[width=1\linewidth]{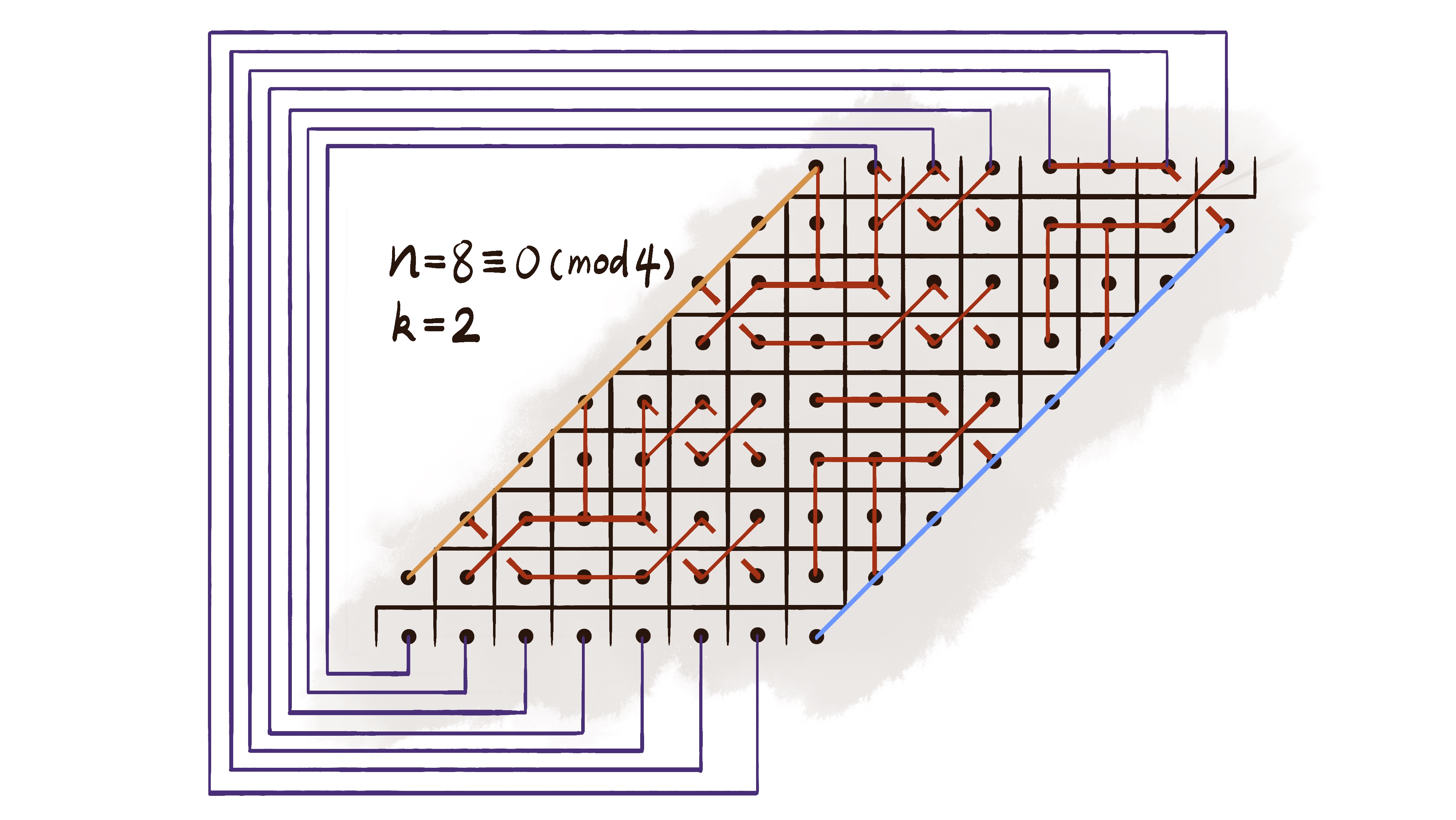}
        \caption{$k=2$, $n = 8 \equiv 0 \pmod{4}$.}
        \label{fig:k2n8}
    \end{figure}
    
\end{enumerate}

\subsection{Construction when $n=3,4,5$}\label{sec:construction:nleq5}

\begin{enumerate}
    \item $n=5$.

     We distinguish two cases depending on whether $k$ is odd or even. 

    \begin{itemize}
        \item $k$ is odd. First, from bottom to top, we take every two rows as a unit and perform two Reidemeister~II moves in each unit. 
        
        When \(k \equiv 3 \pmod{4}\), there remain 10 small squares at the top, forming a four-step stair segment. Applying one loco construction reduces the thickened multicurve to two boundary components. 
        
        When \(k \equiv 1 \pmod{4}\), there remain 6 small squares at the top, forming a three-step stair segment. Performing one Reidemeister~I move and one Reidemeister~II move then reduces the thickened multicurve to two boundary components: one component is a single small square, and the other is the outer part. See Figure \ref{fig:k35n5}.
    
    \begin{figure} [H]
        \centering
        \includegraphics[width=1\linewidth]{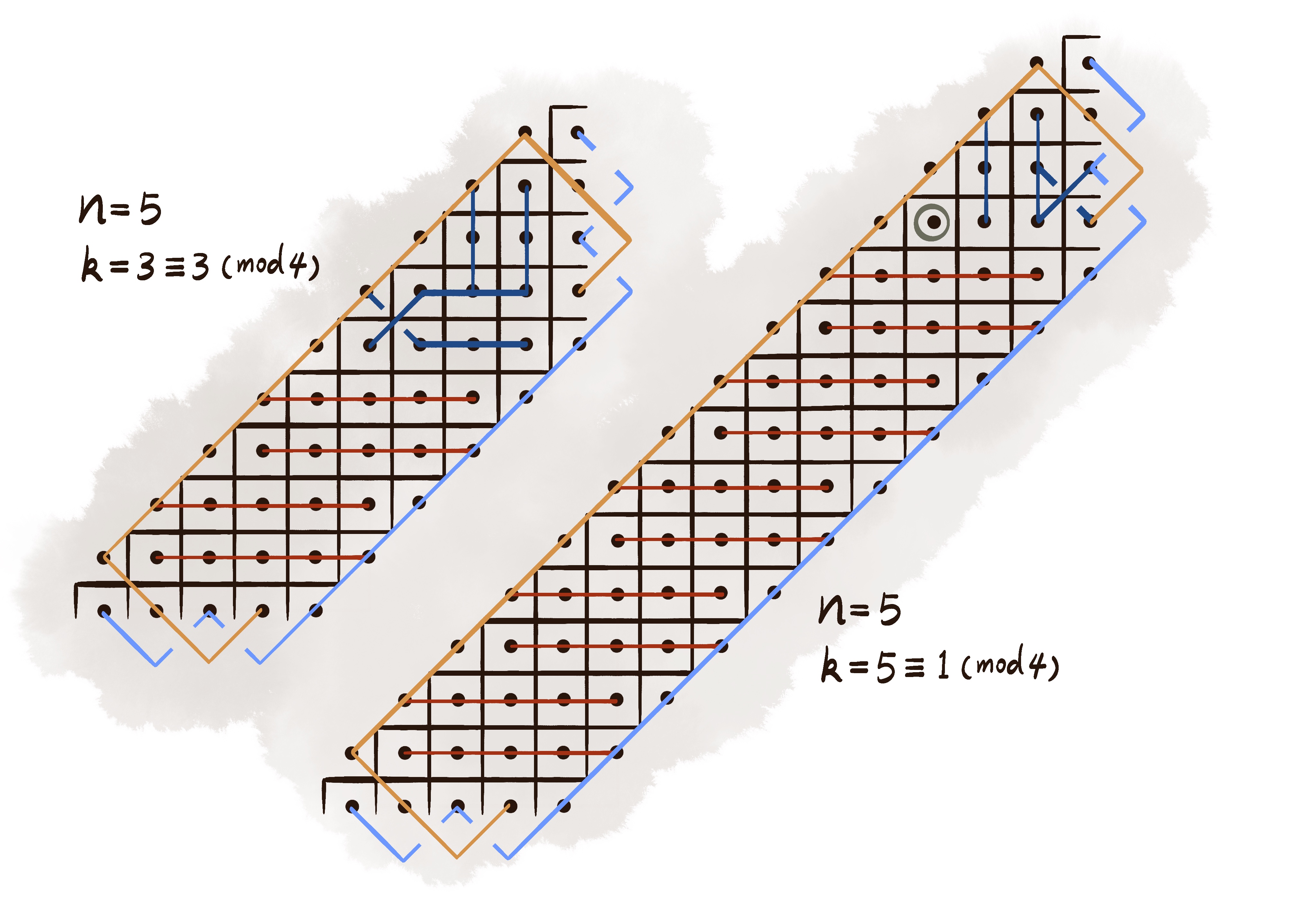}
        \caption{$k = 3 \equiv 3 \pmod{4}$, $n = 5$ and $k = 5 \equiv 1 \pmod{4}$, $n = 5$.}
        \label{fig:k35n5}
    \end{figure}

    \item $k$ is even. We perform two Reidemeister~I moves and four Reidemeister~II moves in $St(5,2)$ as follows to complete the construction when $k=2$, and then concatenate $\frac{k}{2}$ copies of this construction.

    \item $n=4$.

    First, from bottom to top, we take every two rows as a unit and perform a Reidemeister~II move in the middle part of each unit. In the lowest unit we execute two Reidemeister~I moves, while in each of the remaining units we perform one Reidemeister~I move. See Figure \ref{fig:k35n4}.
    
    \begin{itemize}
        \item When $k$ is odd, the resulting thickened multicurve then has two boundary components: one component is a single small square, and the other is the outer part.

        \item When $k$ is even, we add two Reidemeister~I moves on the upper-right side of the resulting thickened multicurve, which then has two boundary components: one component connects three small squares at the top, and the other is the outer part.
    \end{itemize}

    \begin{figure} [H]
        \centering
        \includegraphics[width=1\linewidth]{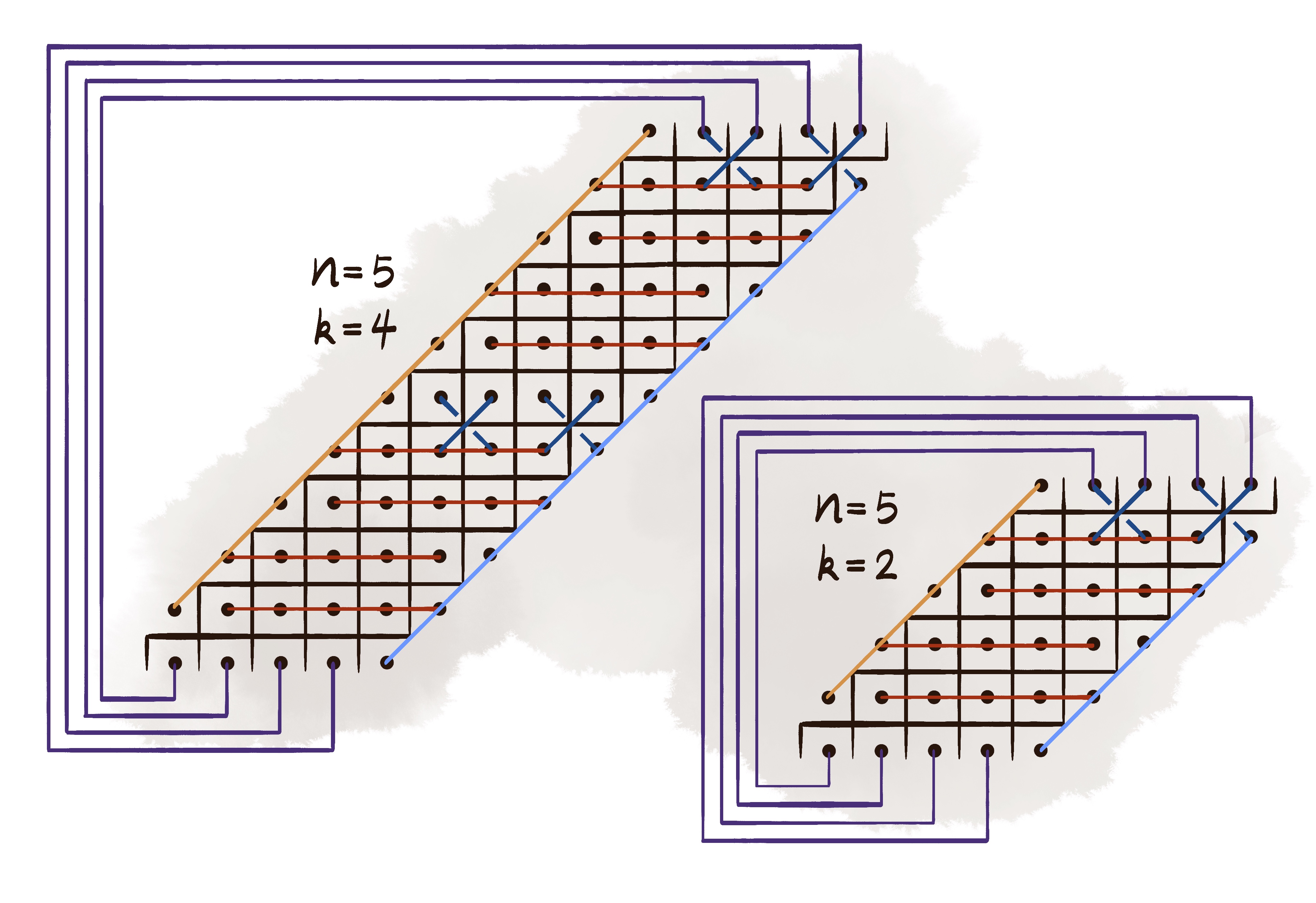}
        \caption{$k = 2$, $n = 5$ and $k = 4$, $n = 5$.}
        \label{fig:k24n5}
    \end{figure}
    
    \end{itemize}

    \begin{figure} [H]
        \centering
        \includegraphics[width=1\linewidth]{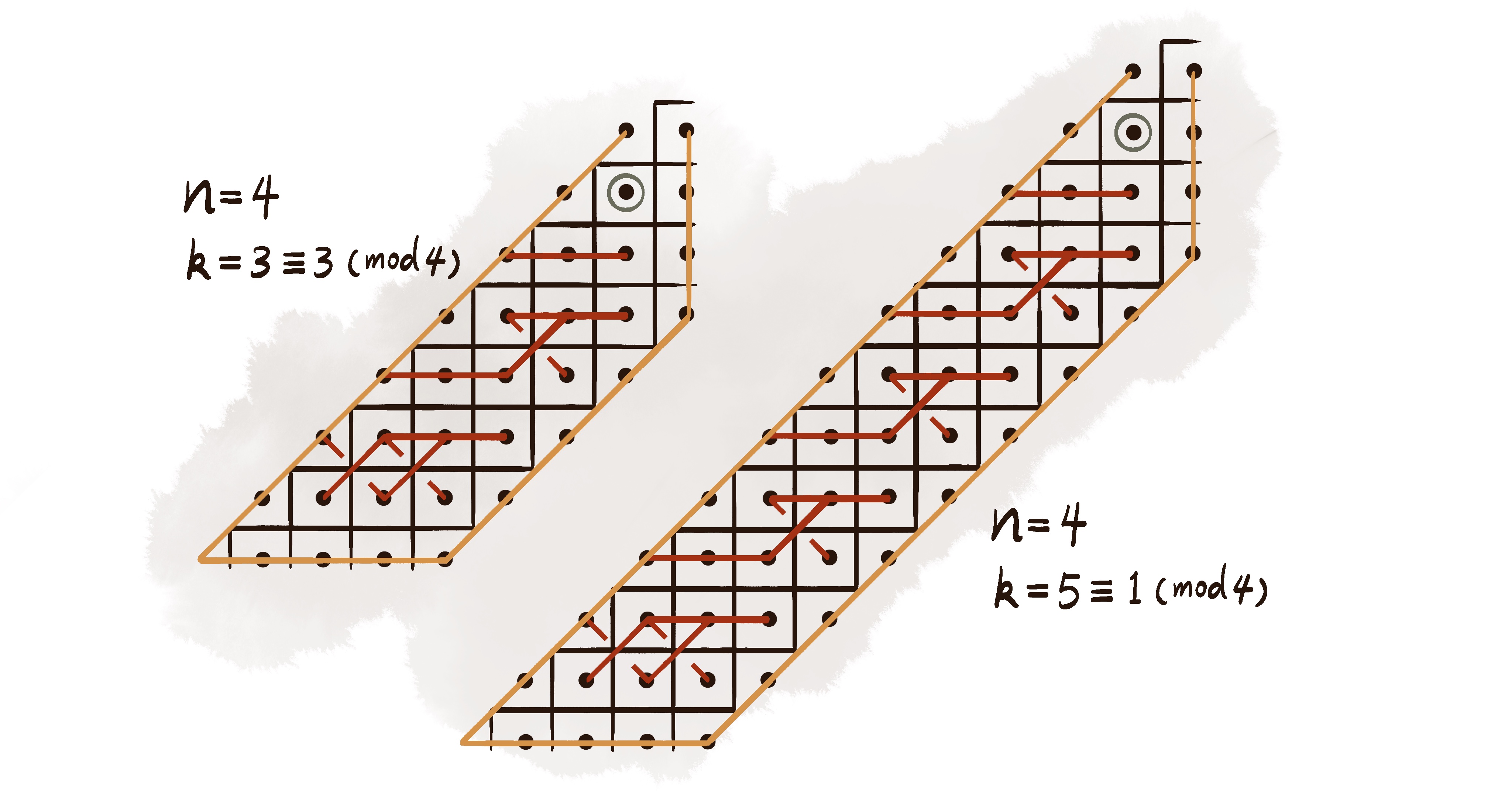}
        \caption{$k = 3 \equiv 3 \pmod{4}$, $n = 4$ and $k = 5 \equiv 1 \pmod{4}$, $n = 4$.}
        \label{fig:k35n4}
    \end{figure}

    \item $n=3$ and $k\neq1$ is odd.

    In this case, each row contains two small squares. At the lower endpoint of the common edge of the two squares in each row except the top two rows, we perform a Reidemeister~I move. For the top two rows, we carry out one Reidemeister~II move. The resulting thickened multicurve then has three boundary components.

    In the final step we carry out the following \emph{special operation} to reduce these three boundary components to one.  

    Let the uppermost intersection point of the stairs be \(P\). We remove the intersection \(P\) and then reinsert it in a different position. After removing \(P\), the thickened multicurve near \(P\) changes from a crossing into two disjoint strips. At this stage the thickened multicurve has two boundary components, coloured red and blue in the middle part of Figure \ref{fig:k3n3}.  

    %Traversing the blue component once in order, we pass successively through points \(1,2,3,4,5,6\). Points \(1,2,5,6\) lie in the middle of the region between the top two rows where the Reidemeister~II move was performed. Points \(3,4\) are located near the former position of \(P\); opposite to them lie points \(A\) and \(B\) on the red component.

    \begin{figure} [H]
        \centering
        \includegraphics[width=1\linewidth]{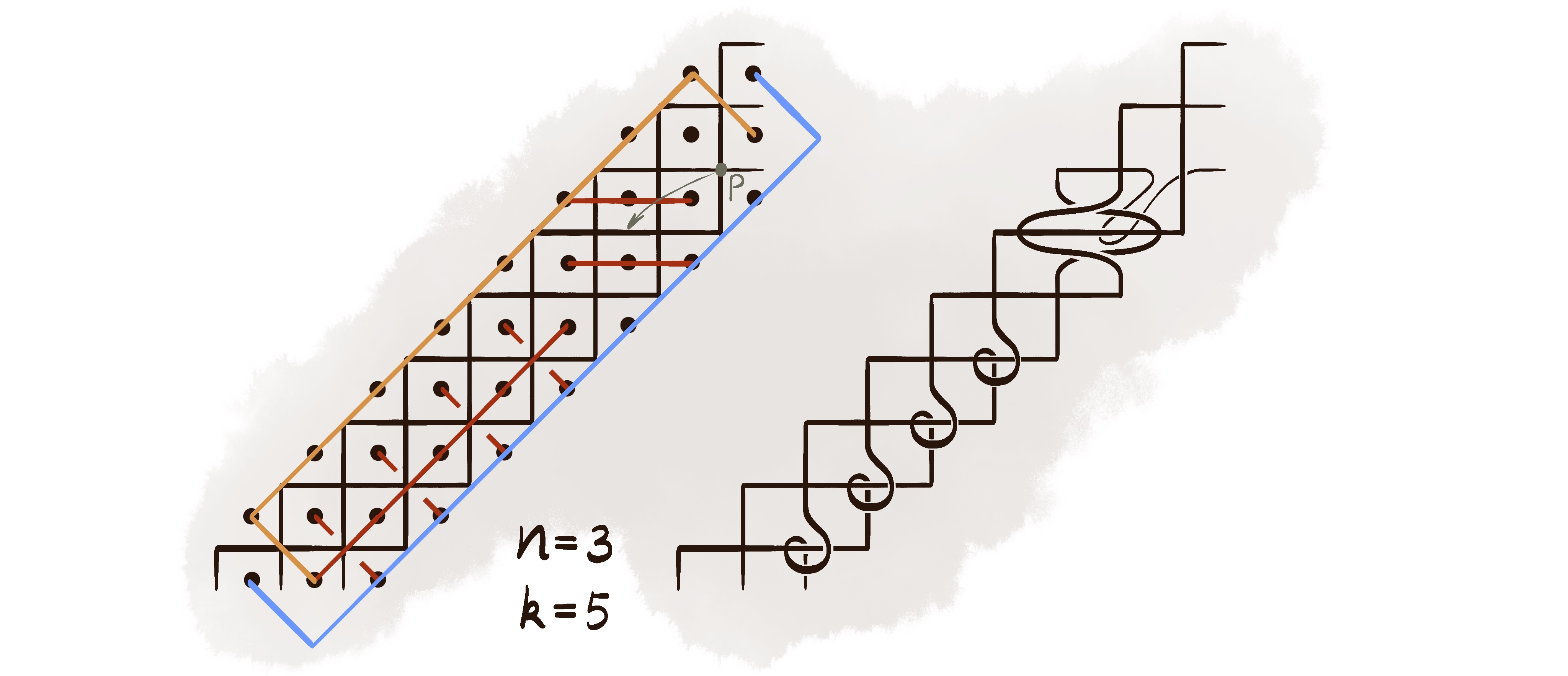}
        \caption{$k = 5$, $n = 3$.}
        \label{fig:k5n3}
    \end{figure}

    We now reconnect the previously removed intersection \(P\) to the region between the top two rows where the Reidemeister~II move was performed(while simultaneously performing the corresponding plumbing surgery on the thickened multicurve). We call this \emph{special operation}.
    
    We prove that this connects the two boundary components of the thickened multicurve into a single component. Before the special operation, the thickened multicurve has two boundary components, coloured red and blue. Traversing the red component, we encounter the points \(A,...,H\) in order; traversing the blue component, we encounter \(a,b,c,d\) in order; see the left part of Figure~\ref{fig:k3n3}. After the special operation, the thickened multicurve has a single boundary component, traversed in the order \(a,b,c, D,E, H, A, B, C,F, G, d\); see the right part of Figure~\ref{fig:k3n3}.

    \begin{figure} [H]
        \centering
        \includegraphics[width=1\linewidth]{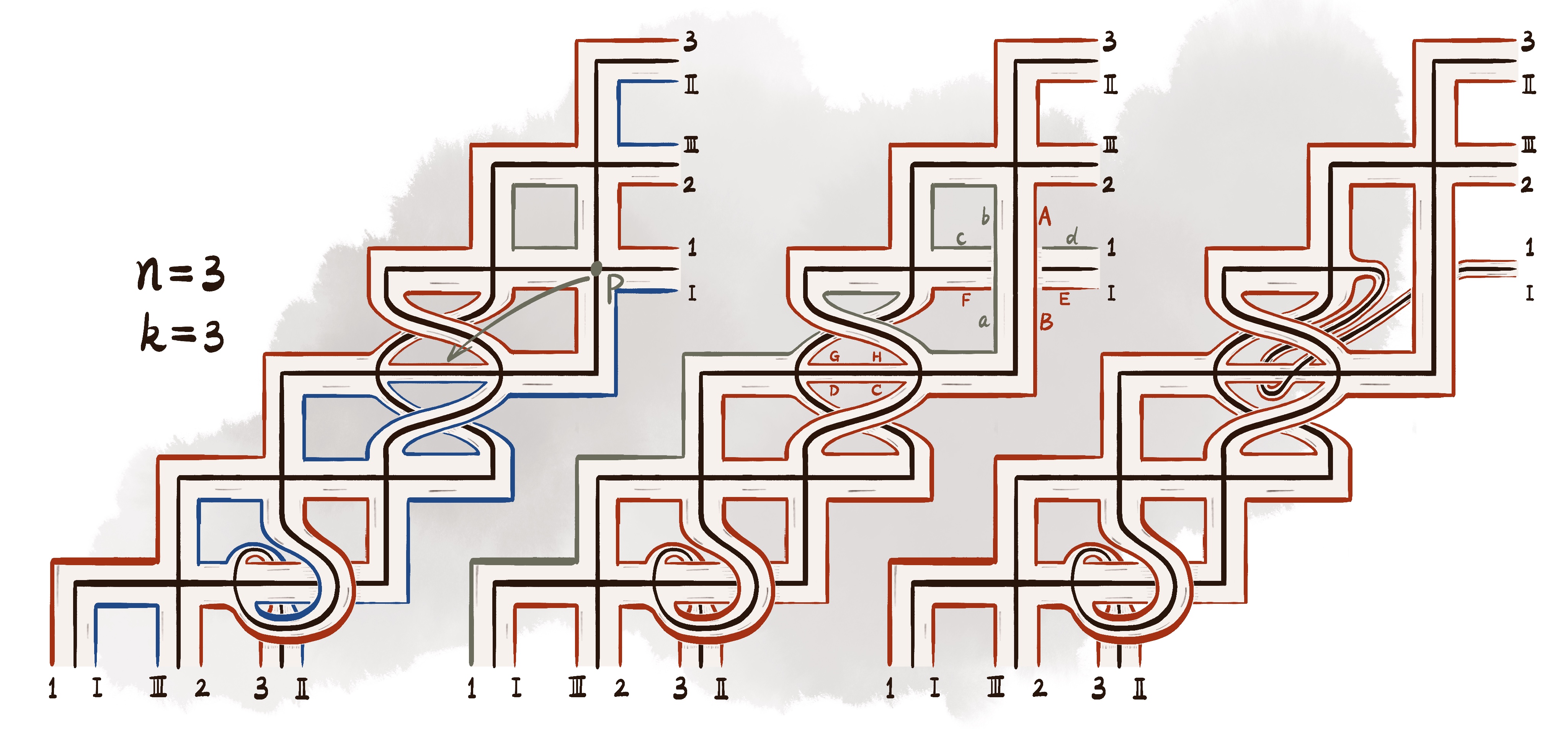}
        \caption{$k = 3$, $n = 3$.}
        \label{fig:k3n3}
    \end{figure}

    \item $n=3$ and $k$ is even.

    When $k = 2$, we perform one Reidemeister~I move and one Reidemeister~II move on $St(3,2)$, and the resulting multicurve has two boundary components. When $k \geq 4$, we concatenate $\frac{k}{2}$ copies of this construction.

    \begin{figure} [H]
        \centering
        \includegraphics[width=1\linewidth]{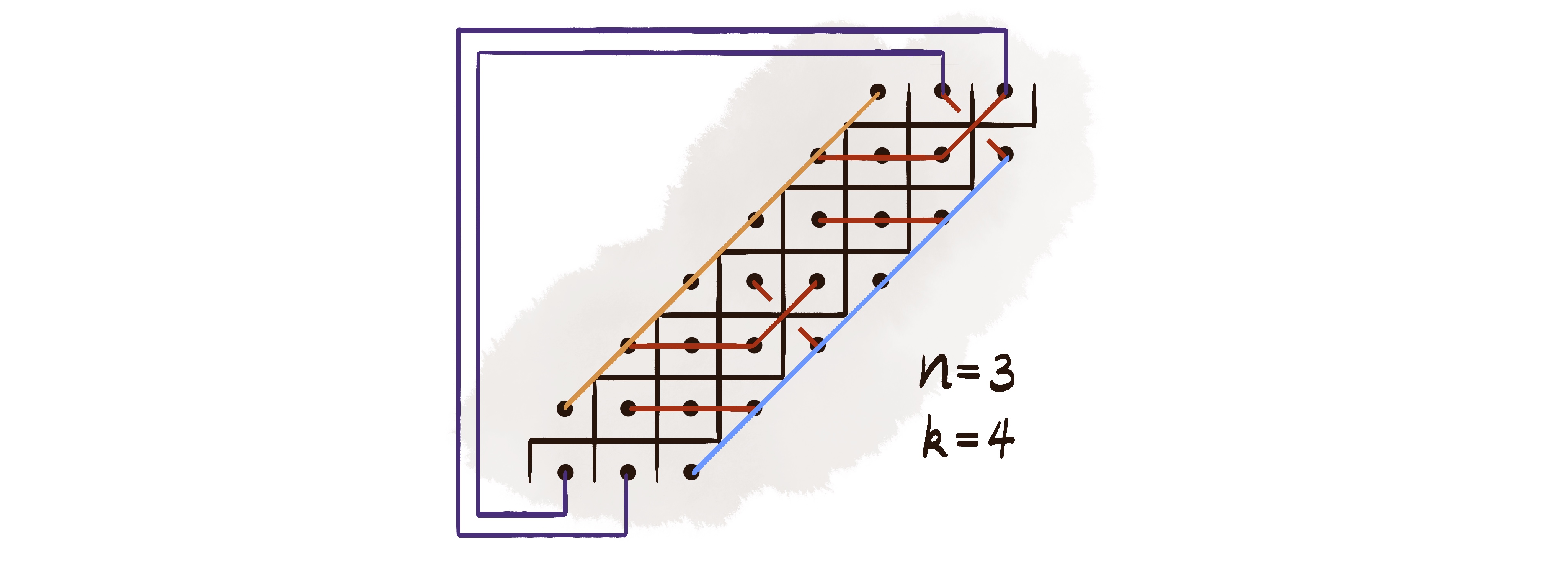}
        \caption{$k = 4$, $n = 3$.}
        \label{fig:k4n3}
    \end{figure}

     \item $n=3$ and $k=1$.

     In this case, after plumbing two annuli, we need to perform an additional plumbing between a third annulus and each of the previous two annuli to obtain the desired thickened multicurve. The plumbing can be carried out in only two ways, as shown in Figure \ref{fig:k1n3_1}. Regardless of which method is chosen, the resulting thickened multicurve has three boundary components.

    \begin{figure} [H]
        \centering
        \includegraphics[width=1\linewidth]{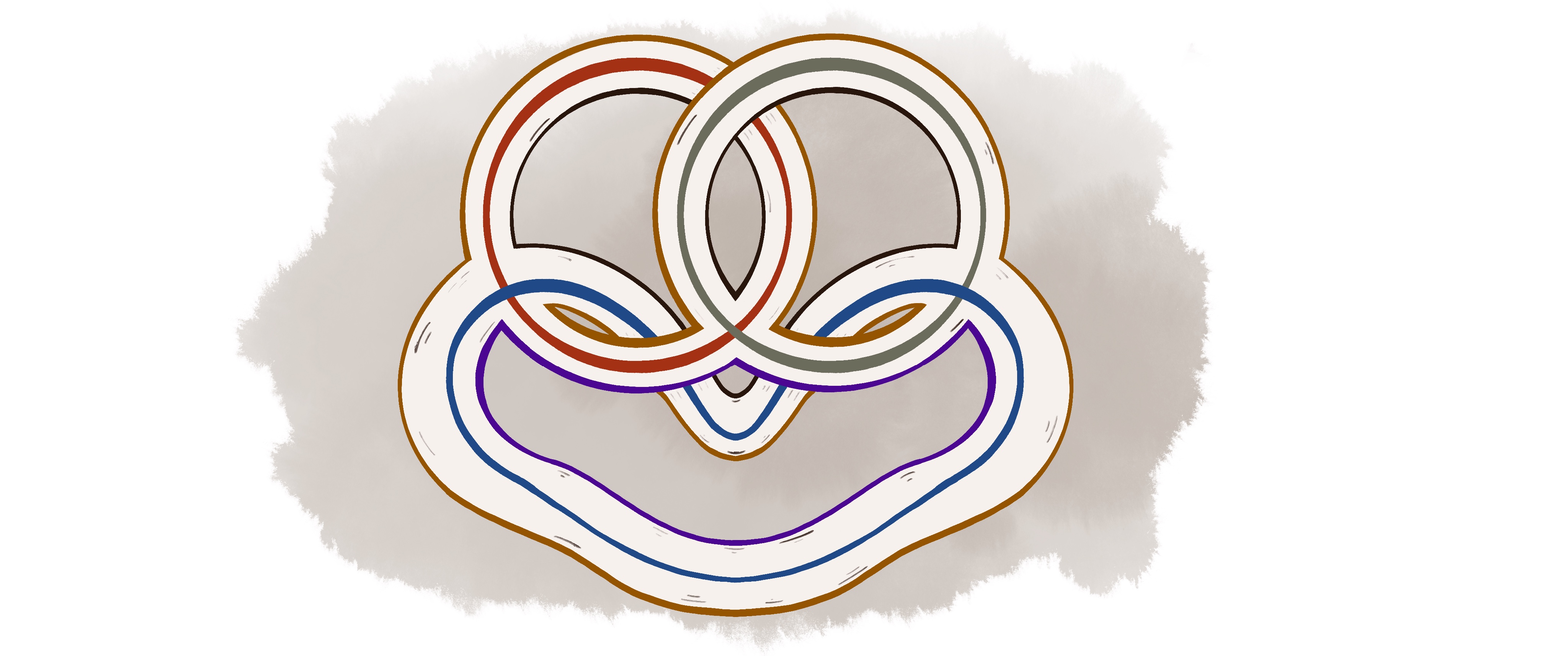}
        \caption{The first case of $k = 1$, $n = 3$.}
        \label{fig:k1n3_1}
    \end{figure}

    \begin{figure} [H]
        \centering
        \includegraphics[width=1\linewidth]{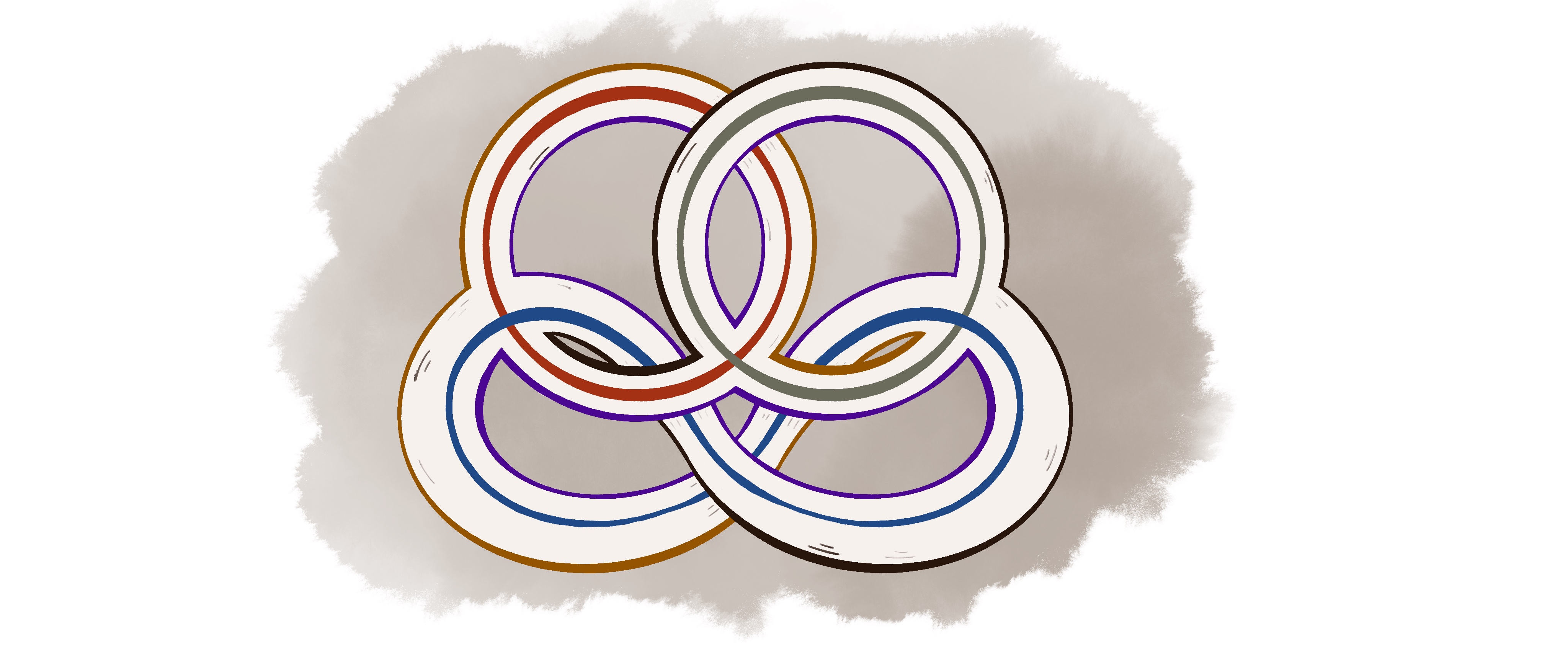}
        \caption{The second case of $k = 1$, $n = 3$.}
        \label{fig:k1n3_2}
    \end{figure}

\end{enumerate}

\section{Proof of \cref{thm:main2}}\label{sec:completion2}

In this section we provide a proof of Theorem \ref{thm:main2}. Note that in the previous two sections we determined only the maximal genus of a surface that admits a filling \(k\)-system with \(n\) curves for given \(n\) and \(k\), but the proof did not give an explicit construction for every possible genus. We fill this gap here. Specifically, we need to show that for every integer $g$ satisfying $g_{k,n-1} < g < g_{k,n}$, there exists a filling $k$-system consisting of exactly $n$ curves on the surface $S_g$. Here $g_{k,n}$ is defined in Theorem \ref{thm:main1}.

\begin{comment}

\subsection{$k$ is even}\label{subsec:k_even}

Similar to Theorem~\ref{theorem:k is even}, we give a construction of a surface of genus $g$, where $g_{k,n-1} < g < g_{k,n}$. Our aim is to construct an attaching graph $G'$ with $n$ vertices such that between any two vertices there are at most $k$ edges; see Section~\ref{subsec:attachinggraph}.  

One method is to attach every curve to the same initial curve using Surgery~I, thereby obtaining a tree with $n$ vertices. Then we perform Surgery~II exactly $g - n + 2$ times, ensuring that for each pair of vertices without an edge the number of connecting edges is at most $k/2$, and for each pair of vertices that already share an edge the number of additional connections is at most $k/2 - 1$. In this way, the total number of edges in $G'$ becomes $g + 1$.

At this stage, the resulting thickened system has exactly two boundary components. Attaching a disk to each boundary component yields a closed surface of genus $g$ that carries a filling $k$-system consisting of $n$ curves.

We only need to prove that the system on $S_g$ is in minimal position, i.e., that no two curves bound a bigon. Consider the tubular neighborhood of any two curves in the thickened system; this defines a subsurface of the thickened system. The boundary of this subsurface contains no bigon, because all orientations of the two curves are coherent. If a bigon existed, the two intersection points corresponding to the bigon would be incoherent, which leads to a contradiction.
\end{comment}

In this subsection, our main goal is to omit a certain number of Reidemeister I and Reidemeister II moves from the construction given in the previous section, while still maintaining a filling \(k\)-system formed by the \(n\) curves. Each omitted Reidemeister I move reduces the genus of the resulting surface by \(1\), and each omitted Reidemeister II move reduces it by \(2\). By omitting a suitable set of such moves, we can lower the genus of the surface and thereby cover every integer genus \(g\) in the range \(g_{k,n-1} < g < g_{k,n}\). More precisely,

\begin{theorem}\label{thm:main3}
    For \(g_{k,n-1} < g < g_{k,n}\), there exists a filling $k$-system with $n$ curves on minimal position.
\end{theorem}

The key point here is that after omitting some Reidemeister~I and Reidemeister~II moves, the curves must still be in minimal position; that is, no two curves may form a bigon. The method we use is to introduce the concept of a "bigon candidate": a set of small squares in \(St(n,k)\) whose union may form a bigon after performing the Reidemeister moves. In Section~\ref{subsec:bigoncandidate} we prove that any bigon must arise from a union of rectangles, and in Section~\ref{subsec:completion} we give a way to omit Reidemeister moves such that all bigon candidates do not form bigons.

\begin{comment}
\begin{definition}[Bigon candidate]
    For a set of small squares in $St(n,k)$, we call this set a \emph{bigon candidate} if the union of the squares satisfies:
    \begin{itemize}
        \item Its boundary is contained in two of the $n$ closed curves.
        \item After applying the Reidemeister moves, these boundary curves form the boundary of a bigon, and the bigon consists precisely of the squares in this candidate set.
    \end{itemize}
\end{definition}
\end{comment}

\subsection{Bigon candidates}\label{subsec:bigoncandidate}

For the original stair $St(n,k)$ with no Reidemeister moves, every pair of curves intersects \(k\) times, which naturally forms \(k-1\) bigons between them, numbered from bottom to top as \(B_1,\dots,B_{k-1}\). The intersection point between two adjacent bigons \(B_i\) and \(B_{i+1}\) is denoted by \(X_i\). If \(k\) is even, there is an additional bigon, which we denote by \(B_k\), whose boundary is contained in the same two curves and consists of two arcs meeting at the intersection points \(X_k\) and \(X_1\), connecting the top and the bottom of \(St(n,k)\). In other words, all these bigons form a line when \(k\) is odd, and form a cycle when \(k\) is even. %We define the bigon candidate as follows:

\begin{definition}[Bigon candidate]
    A \emph{bigon candidate} is a union of a consecutive subsequence of \(B_1,\dots,B_{k-1}\) (and possibly including \(B_k\) when \(k\) is even), whose boundary consists of arcs from exactly two curves. Moreover, when \(k\) is even, the bigons \(B_1,\dots,B_k\) form a cycle; consequently, the complement of a consecutive subsequence (i.e., the union of \(B_{j+1},\dots,B_k,B_1,\dots,B_i\)) is also a bigon candidate.
    
    %For stairs, a bigon candidate is the union of rectangles whose sides are from the same two curves that connect at their corners. Note that some of the rectangles may contain the outside components of the stair, as shown in the figure.

\end{definition}

%\textbf {To Xiao Chen, please draw examples of bigon candidates.}

     \begin{figure} [H]
        \centering
        \includegraphics[width=1\linewidth]{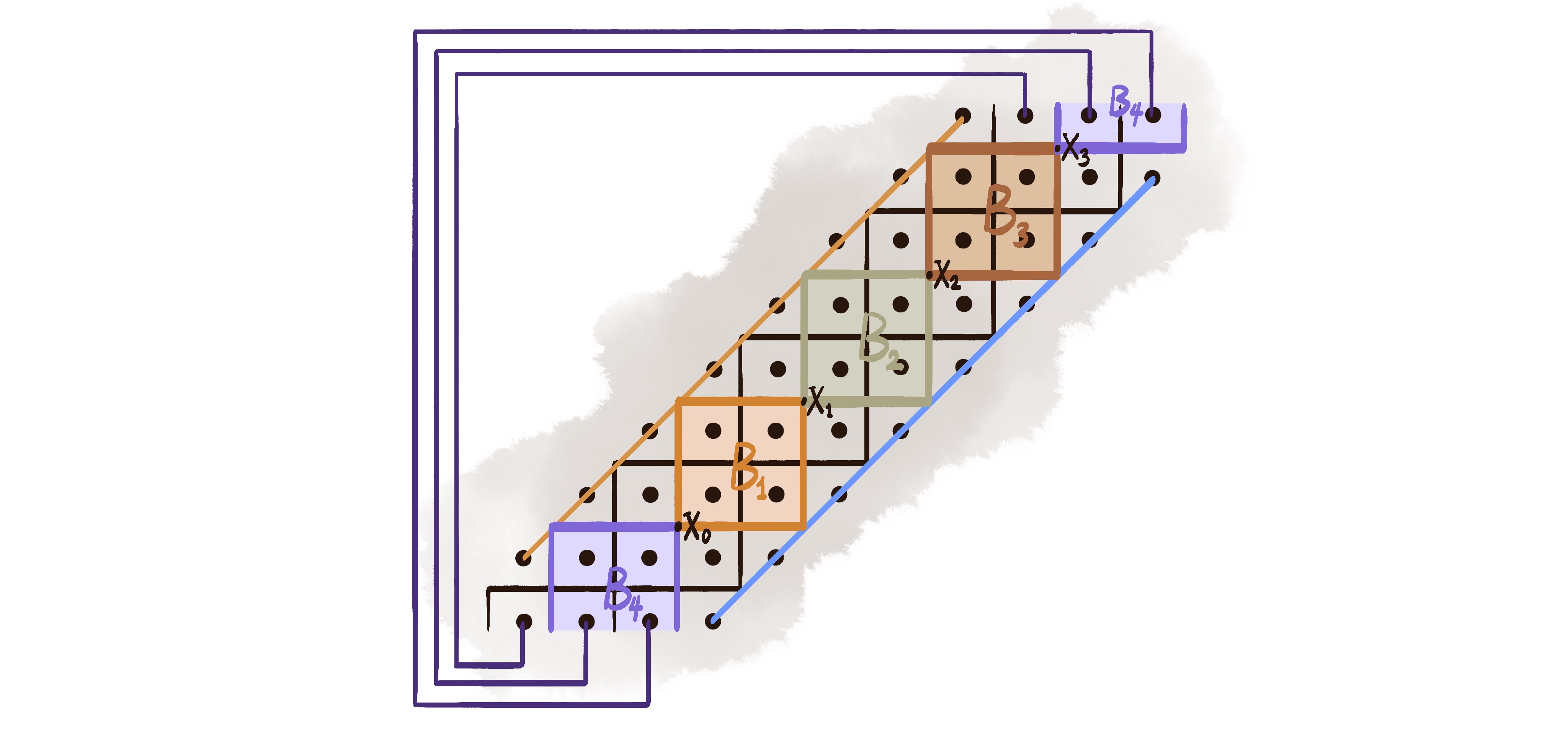}
        \caption{Bigon candidates.}
        \label{fig:BigonCandidates1}
    \end{figure}

The following lemma describes all possible situations in which a bigon could arise:

\begin{comment}
\begin{lemma}
    Let \(\gamma_1,\gamma_2\) be two curves of \(St(n,k)\) that enclose several rectangles on \(St(n,k)\), as indicated by the red rectangles in the figure. After performing a number of Reidemeister~I and Reidemeister~II moves, if for each bigon candidate bounded by the two curves, at least one Reidemeister~I or Reidemeister~II moves have been applied inside these rectangles, or apart from possible Reidemeister~I moves at the crossing points of the rectangles,or Reidemeister~I or Reidemeister~II move crosses the boundaries of the rectangles, then \(\gamma_1\) and \(\gamma_2\) do not form a bigon.
\end{lemma}
\end{comment}

\begin{lemma}\label{lem:candidate1}
    If there is a bigon after the surgeries, it must occur at one of the bigon candidates of the original stair.
\end{lemma}

\begin{proof}
      If we perform a Reidemeister~I move at a crossing point \(X_i\) (marked by a blue cross in the figure), the two adjacent bigons \(B_i\) and \(B_{i+1}\) merge into a single bigon.

    % \begin{figure} [H]
    %     \centering
    %     \includegraphics[width=1\linewidth]{Pictures/ReidemeisterImove.jpg}
    %     \caption{Reidemeister I move.}
    %     \label{fig:ReidemeisterImove}
    % \end{figure}

 Suppose that after performing all Reidemeister~I and Reidemeister~II moves, there exists a bigon such that two curves intersect at \(X_i\) and \(X_j\). Then no Reidemeister~I move has been performed at \(X_i\) or \(X_j\). Observe that if in every Reidemeister move we band distinct boundary components (which is true in our case), we never create new boundary components. Therefore this bigon must be one of our bigon candidates.
\end{proof}

\begin{lemma}\label{lem:candidate2}
    If there is a bigon after the surgeries, it must occur at one of the bigon candidates of the original stair. Moreover, given a bigon candidate, if at least one Reidemeister~I or Reidemeister~II move is applied inside these rectangles, or the crossing of a Reidemeister~I move is at the corner of one rectangle other than the intersections, or a Reidemeister~I or Reidemeister~II move crosses the boundaries of the rectangles after the surgeries, then the bigon candidate is not a bigon.
\end{lemma}

\begin{proof}

%定义一个“手术痕迹”，在bigon candidate（最终形成bigon的小方块的并集）里的，在Lemma 5.1之前。避免引入手术顺序的概念，做的手术是原有手术的子集。

   We observe that no Reidemeister move can cross the boundary between the interior of \(\bigcup_{s=i+1}^{j} B_s\) and the exterior (as configurations \(A\) and \(B\) in Figure \ref{fig:BigonCandidates2}), because such a move would connect the interior of \(\bigcup_{s=i+1}^{j} B_s\) with its exterior into one component, and a bigon must be a closed region.

 Secondly, there cannot exist a Reidemeister~II move that connects two sets of three squares lying respectively inside and outside \(\bigcup_{s=i+1}^{j} B_s\) (as in configuration \(C\) in Figure \ref{fig:BigonCandidates2}), nor can there be any Reidemeister move entirely inside \(\bigcup_{s=i+1}^{j} B_s\) (as in configurations \(D\) and \(E\) in Figure \ref{fig:BigonCandidates2}). This is because such a move would band distinct boundary components and increase the genus inside \(\bigcup_{s=i+1}^{j} B_s\); consequently, after all Reidemeister moves have been performed, the region cannot remain a bigon. 
 
    \begin{figure} [H]
        \centering
        \includegraphics[width=1\linewidth]{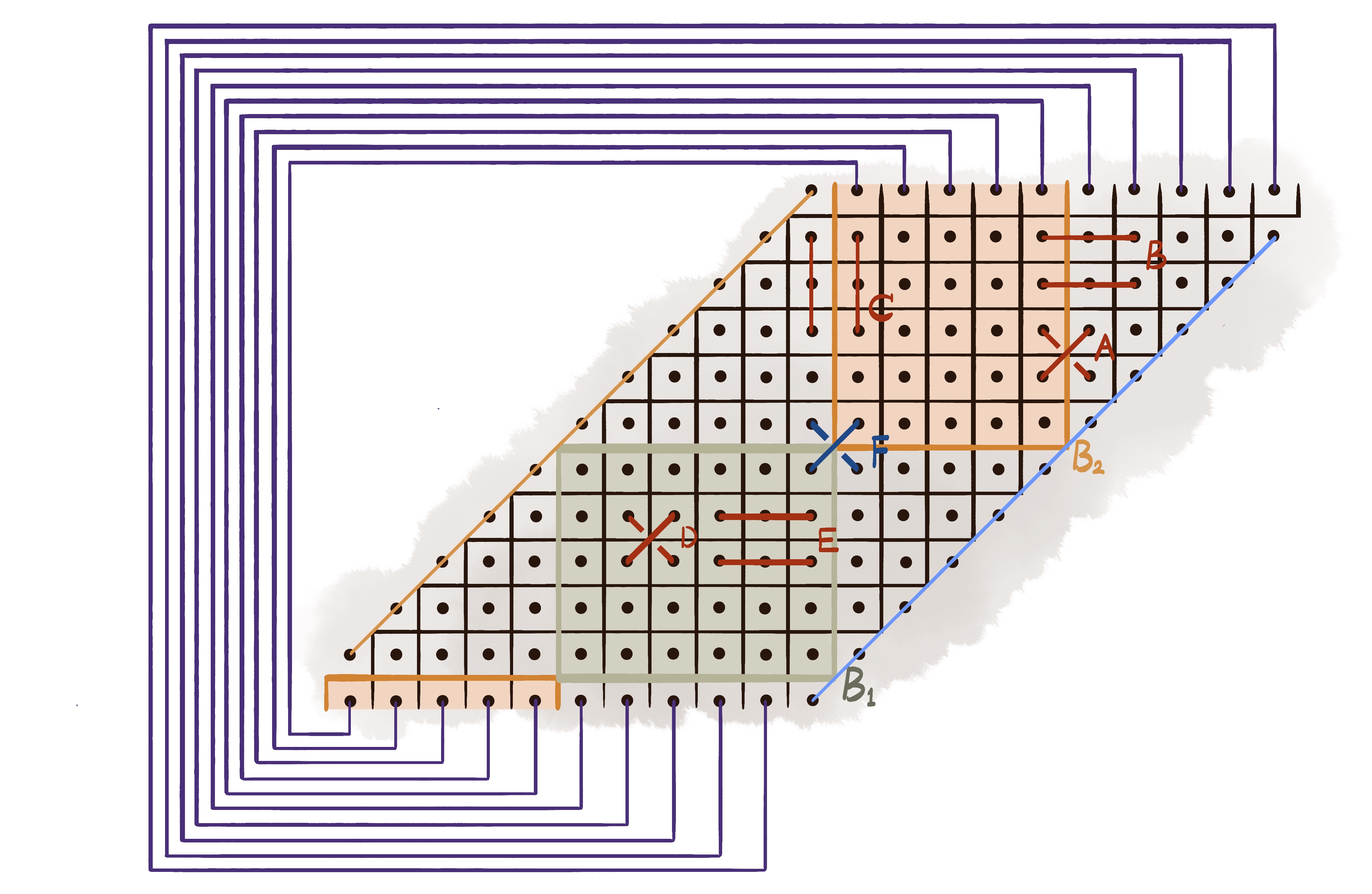}
        \caption{Bigon candidates and configurations $A,B,C,D,E,F$.}
        \label{fig:BigonCandidates2}
    \end{figure}
    
Therefore, in the final configuration no Reidemeister move crosses the boundaries of the rectangles except possibly Reidemeister~I moves applied at the intersection points \(X_{i+1},\dots,X_{j-1}\) (as in configuration \(F\) in Figure \ref{fig:BigonCandidates2}).

\end{proof}

\subsection{Proof completion}\label{subsec:completion}

Next, based on the construction given in Section~\ref{sec:kodd} and Section~\ref{subsec:bigoncandidate}, we remove some Reidemeister moves that satisfy the conditions stated in Lemma~\ref{lem:candidate2}, in order to prove Theorem~\ref{thm:main3}, which ensures that a corresponding construction exists for every possible genus.

The proof is divided into two parts, depending on whether the number of curves \(n\) is large or small. When \(n\) is large, the case is easy, since we only need to omit some Reidemeister moves in certain loco operations and carriage operations; the number of Reidemeister moves is large enough that the possible reduction in genus exceeds \(g_{k,n} - g_{k,n-1}\). When \(n\) is small, the number of Reidemeister moves is insufficient to achieve the required genus reduction, and a more detailed case analysis is needed. 

In the second part, we employ an inductive argument on \(k\). Observe that a \((k-1)\)-system is naturally also a \(k\)-system. We have already shown that for any genus \(g \le g_{k-1,n}\), there exists a filling \((k-1)\)-system consisting of at most \(n\) curves on \(S_g\). Hence it suffices to consider the range \(g_{k-1,n} < g < g_{k,n}\). Since the difference \(g_{k,n} - g_{k-1,n}\) is bounded above by \(k\), only a limited reduction in genus is required, which can be achieved by omitting a small number of moves. This is feasible because \(k\) itself is small.

%Before proving the theorem, to simplify the argument, we first give a lemma, which states that 

\begin{proof}[Proof of Theorem \ref{thm:main3}]
    In the case $k = 1$, there are no bigon candidates; thus omitting any subset of Reidemeister moves yields a system with no bigons. Hence we only need to consider $k \geq 2$. We divide the proof into two cases based on the number $n$ of curves. The first is $n\geq8$, and the second is $3\leq n\leq 7$.

\begin{itemize}
    \item  When $n\geq8$, according to the construction in Section~\ref{sec:kodd} and Section~\ref{subsec:bigoncandidate}, if \(k\) is odd, the stairs \(St(n,k)\) consist of \(\frac{k-1}{2}\) basic blocks, and each basic block contains at least \(\lceil\frac{n-3}{4}\rceil\) four-row blocks in which the basic construction is performed; if \(k\) is even, the stairs \(St(n,k)\) consist of \(\frac{k}{2}\) basic blocks. In each four-row block, we may omit a subset of the Reidemeister moves from the middle carriage operation and the right loco operation, while keeping the left loco operation unchanged. Note that if we omit a Reidemeister~I move, the genus decreases by \(1\); if we omit a Reidemeister~II move, the genus decreases by \(2\). In this case, when $n\geq8$ and $k\geq 2$, the maximal possible reduction in genus is no less than
    \begin{equation}
        \frac{k-1}{2}\cdot \lceil\frac{n-3}{4}\rceil\cdot (5+2(n-6))\geq \frac{1}{2}k(n-1)\geq g_{k,n}-g_{k,n-1}-1.
    \end{equation}

    We prove that omitting any subset of these Reidemeister moves still yields a multicurve with no bigons. By Lemma~\ref{lem:candidate1}, any bigon must arise from some bigon candidate, that is, a single rectangle or a union of rectangles whose sides belong to the same two curves and meet at their corners. A single rectangle cannot occur, because all small squares at the upper-left corners in the stairs participate in Reidemeister moves.

        \begin{figure} [H]
        \centering
        \includegraphics[width=1\linewidth]{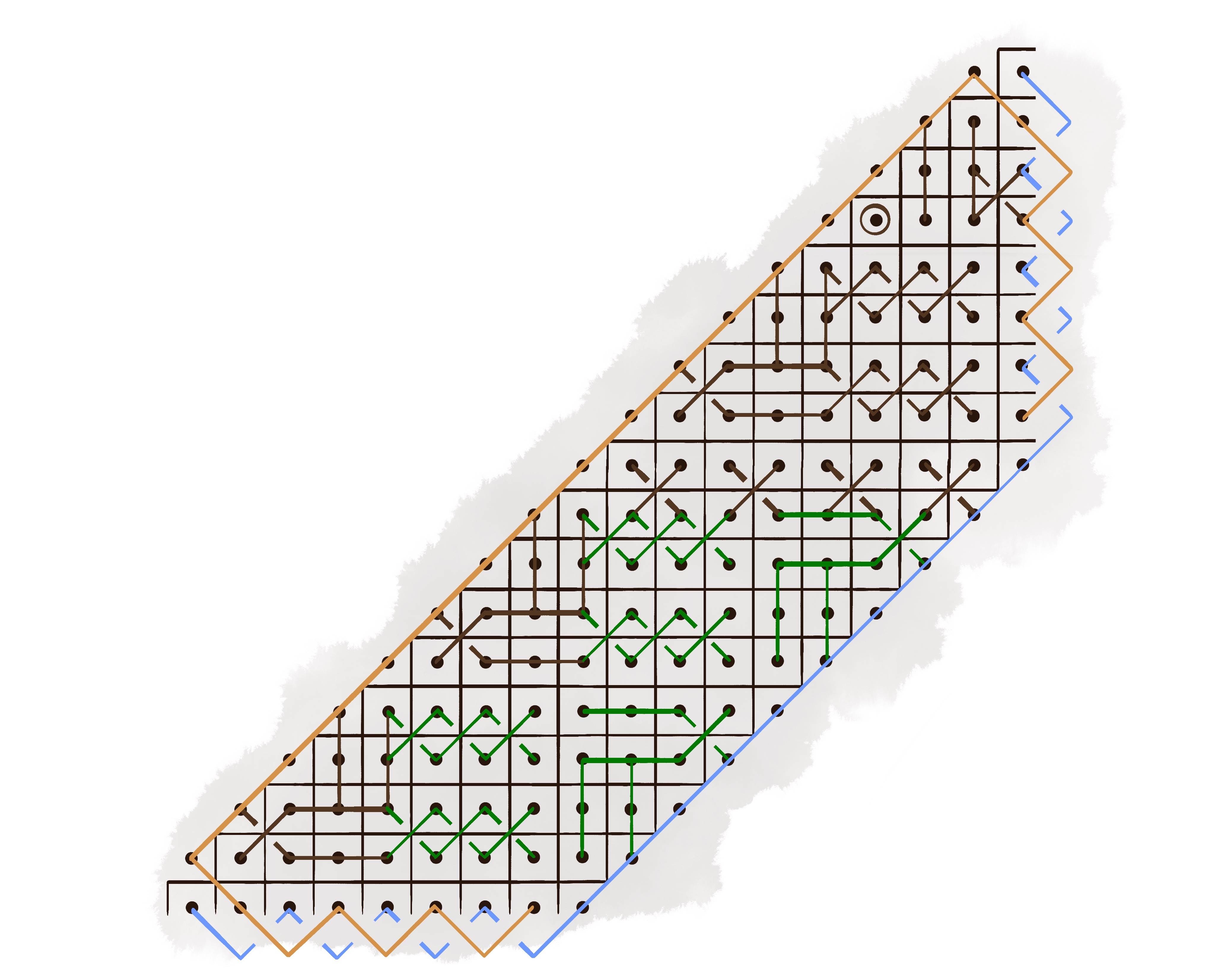}
        \caption{Omitting Reidemeister moves when $k=3$ and $n=9$. The Reidemeister moves that can be omitted are in the right loco operations and the carriage operations in the four-row blocks, which are labelled in green.}
        \label{fig:ReidemeisterMovesOmitting}
    \end{figure}

    It remains to show that a union of rectangles is also impossible. First assume that $k$ is odd. Suppose the union is \(\bigcup_{s=i+1}^{j} B_s\) with crossing points \(X_{i+1},\dots,X_{j-1}\). Then, as shown in Lemma~\ref{lem:candidate2}, no Reidemeister move crosses the vertical and horizontal lines passing through \(X_{i+1}\) except possibly a Reidemeister~I move at \(X_{i+1}\) itself. In this situation, the horizontal line must be the bottom line of some left-hand loco operation, and the vertical line must be the right line of some left-hand loco operation. Consequently, some Reidemeister move would have to cross the indicated line inside the rectangle \(B_{i+2}\), which is impossible. Hence such a configuration cannot exist. If $k$ is even, although all the bigons form a cycle,  the same proof strategy applies.
    
    \item  When $n \leq 7$, we make use of the following induction argument: for any $g \leq g_{k-1,n}$, there exists a filling $(k-1)$-system consisting of at most $n$ curves on $S_g$. A $(k-1)$-system is also naturally a $k$-system, so it suffices to consider the range $\max\{g_{k-1,n}+1,g_{k,n-1}+1\} \leq g \leq g_{k,n}-1$. The number $g_{k,n}-g_{k-1,n}-1$ equals $10,7,4,2,1$ when $n=7,6,5,4,3$. Note that when $n=7,6,5,4$, the construction is the same regardless of the parity of $k$; it differs only when $n=3$.

    When \(n = 7\), we may omit the two Reidemeister~I moves in the carriage operation, the Reidemeister moves in the right loco operation within the basic construction in the bottom four rows, and the three rightmost Reidemeister~I moves at the top of the bottom Construction~II. This reduces the genus by \(10\). If \(1 \leq g_{k,n} - g \leq 9\), we omit a suitable subset of the Reidemeister moves listed above.

    When \(n = 6\), we may omit the Reidemeister moves in the right loco operation within the basic construction in the bottom four rows, together with a Reidemeister~II move above it. This reduces the genus by \(7\). If \(1 \leq g_{k,n} - g \leq 6\), we omit a suitable subset of the Reidemeister moves listed above.

    When \(n = 5\), we may omit the unique Reidemeister~I move at the top and two Reidemeister~II moves  below it, specifically the two on the right-hand side in two different two-row units. This reduces the genus by \(5\). If \(1 \leq g_{k,n} - g \leq 4\), we omit a suitable subset of the Reidemeister moves listed above. Note that we must keep the unique Reidemeister~I move, because omitting only Reidemeister~II moves reduces the genus by an even number, which cannot achieve a reduction of \(1\) or \(3\).

    When \(n = 4\) and \(g_{k,n} - g = 2\), we omit the bottommost Reidemeister~II move; when \(g_{k,n} - g = 1\), we omit the leftmost Reidemeister~I move. 
    
    When \(n = 3\) and \(g_{k,n} - g = 1\), we reduce the genus by 1 by omitting the special operation when $k$ is odd and replacing the bottom Reidemeister~II move with a Reidemeister~I move when $k$ is even.
\end{itemize}
    
\end{proof}

%==============================%
%\newpage

% \bibliographystyle{acm}
% \bibliography{reference.bib}
\printbibliography[title={References}] %打印参考文献

%\end{multicols}

\end{document}